\documentclass[12pt]{article}
\usepackage{amssymb,theorem,amsmath}
\usepackage{graphicx}
\usepackage{pictex}

\numberwithin{equation}{section}

\def\csect#1{Section~\ref{#1}}
\def\capp#1{Appendix~\ref{#1}}
\def\csects#1#2{Sections~\ref{#1} and \ref{#2}}
\def\csectt#1#2{Sections~\ref{#1}--\ref{#2}}
\let\Prp=\Pr \def\Pr{\Prp\nolimits}
\newcommand{\be}{\begin{equation}}
\newcommand{\ee}{\end{equation}}

\newtheorem{theorem}{Theorem}[section]
\newtheorem{corollary}[theorem]{Corollary}
\newtheorem{lemma}[theorem]{Lemma}

\newtheorem{conjecture}[theorem]{Conjecture}
{\theorembodyfont{\rm} 
\newtheorem{remark}[theorem]{Remark}
\newtheorem{definition}[theorem]{Definition}
}

\newcommand{\cthm}[1]{Theorem~\ref{#1}}
\newcommand{\cthms}[2]{Theorems~\ref{#1} and \ref{#2}}

\newcommand{\clem}[1]{Lemma~\ref{#1}}
\newcommand{\cdef}[1]{Definition~\ref{#1}}

\newcommand{\ccor}[1]{Corollary~\ref{#1}}
\newcommand{\cfig}[1]{Figure~\ref{#1}}
\newcommand{\crem}[1]{Remark~\ref{#1}}
\newcommand{\cconj}[1]{Conjecture~\ref{#1}}

\newenvironment{proof}{\noindent\textbf{Proof:} }
   {\hfill {\quad$\blacksquare$}\par\medskip}
\newenvironment{proofof}[1]{\medskip\noindent
   \textbf{Proof of #1:} }{\hfill {\quad$\blacksquare$}\par\medskip}

    \def\bbz{\mathbb{Z}}

    \def\bbr{\mathbb{R}}
    \def\bbn{\mathbb{N}}
    \def\bbj{\mathbb{J}}
    \def\bbjb{\overline{\mathbb{J}}}

\def\n{{\bf n}}\def\m{{\bf m}}\def\i{{\bf i}}
\def\L{{\bf L}}
  \def\P{{\cal P}}

\def\interiorM{\rlap{\raise9pt\hbox to9.5pt{\hss$\scriptscriptstyle\circ$}}M}
\def\Fc{{\cal F}}\def\Fch{\widehat\Fc}
\def\S{{\cal S}}\def\L{{\cal L}}
\def\M{{\cal M}}\def\Mb{\overline\M}\def\T{{\cal T}}
\def\N{{\cal N}}\def\Nb{\overline\N}

\def\tauhat{\hat\tau}
\def\Xh{\widehat X}\def\Yh{\widehat Y}\def\Fh{\widehat F}
\def\Ph{\widehat P}\def\Ah{\widehat A}
\def\Xt{\widetilde X}\def\Pt{\widetilde P}
\def\Qh{\widehat Q}\def\Ahat{\widehat A}
\def\Qo{\overline Q}\def\Yo{\overline Y}\def\Xo{\overline X}

\def\tauo{\overline\tau}\def\nuo{\overline\nu}

\def\rhohat{\hat\rho}\def\pihat{\hat\pi}

\long\def\kill#1\endkill{\relax}

\def\0{{\it0}}\def\1{{\it1}}\def\2{{\it2}}\def\3{{\it3}}
\def\rng#1#2{\hbox{$(#1\!:\!#2)$}}
\newdimen\mbsize \mbsize=\hsize \multiply\mbsize by 4  \divide\mbsize by 5

\def\muLN{\mu^{(N,L)}}\def\muLz{\mu^{(\zeta,L)}}\def\muz{\mu^{(\zeta,\infty)}}
\def\ZLN{{\cal Z}_{N,L}}

\def\phileft{\phi_{\rm left}}
\def\phiright{\phi_{\rm right}}

\def\muhat{\hat\mu}\def\muhattwo{\muhat^{(2)}}
\def\nuhat{\hat\nu}
\def\Xha{\Xh^*}
\def\XhaL{\Xh^{*(L)}}\def\XhaNL{\Xh^{*(N,L)}}
\def\XhaK{\Xh^{*(K)}}
\def\Psih{\widehat\Psi}
\def\ind{{\bf1}}

\begin{document}

\title{\vskip-0.2truein Stationary States of the One-Dimensional
Discrete-Time Facilitated Symmetric Exclusion Process\footnote{Dedicated
  to the memory of Freeman Dyson, friend and teacher.}}

\author{S. Goldstein\footnote{Department of Mathematics,
Rutgers University, New Brunswick, NJ 08903.},
J. L. Lebowitz\footnotemark[1],
\footnote{Also Department of Physics, Rutgers.}\ \ 
and E. R. Speer\footnotemark[1]}
\date{January 12, 2022}
\maketitle

\begin{flushleft}\noindent {\bf Keywords:} Symmetric facilitated exclusion
processes, symmetric stack model, one dimensional conserved lattice gas,
facilitated jumps, translation invariant steady states, F-SSEP
\end{flushleft}


\begin{abstract} We describe the extremal translation invariant
stationary (ETIS) states of the facilitated exclusion process on $\bbz$.
In this model all particles on sites with one occupied and one empty
neighbor jump at each integer time to the empty neighbor site, and if two
particles attempt to jump into the same empty site we choose one randomly
to succeed. The ETIS states are qualitatively different for densities
$\rho<1/2$, $\rho=1/2$, and $1/2<\rho<1$, but in each density region we
find states which may be grouped into families, each of which is in
natural correspondence with the set of all ergodic measures on
$\{0,1\}^\bbz$.  For $\rho<1/2$ there is one such family, containing all
the ergodic states in which the probability of two adjacent occupied
sites is zero.  For $\rho=1/2$ there are two families, in which
configurations translate to the left and right, respectively, with
constant speed 2.  For the high density case there is a continuum of
families.  We show that all ETIS states at densities $\rho\le1/2$ belong
to these families, and conjecture that also at high density there are no
other ETIS states.  We also study the possible ETIS states which might
occur if the conjecture fails.  \end{abstract}

\section{Introduction\label{intro}}

The {\it facilitated symmetric simple exclusion process} (F-SSEP) is a
model of particles moving on a lattice, in which a particle can jump to a
neighboring (empty) site only if another of its neighboring sites is
occupied (by a {\it facilitating} particle).  In this paper we consider
only the case of synchronous discrete-time dynamics on the
one-dimensional lattice $\bbz$, except that in \crem{ulmu} below we
discuss briefly the situation for one-dimensional continuous time
dynamics.  For further results on the one-dimensional case, see
\cite{AGLS,bbcs,BM,BESS,BES,CZ,GLS2,gkr,GR,Oliveira,ST}; for results of
simulations of the continuouis-time model in higher dimensions see
\cite{hl,mcl,rpv}.

The configuration space of the model is $X=\{0,1\}^\bbz$; if $\eta$ is a
configuration in $X$ then we say that a site $i$ with $\eta(i)=1$ is {\it
occupied} by a particle, and a site with $\eta(i)=0$ is {\it unoccupied}
or {\it empty}.  The (stochastic) dynamics is defined as follows: if
$\eta_t$ is the configuration at time $t$, $t\in\bbz$, then each particle
in $\eta_t$ with exactly one occupied neighboring site attempts to jump
to its unoccupied neighboring site; the jump takes place unless two
particles attempt to jump on the same site, in which case one of them is
chosen at random to succeed, with each choice equally likely.
$\eta_{t+1}$ is the resulting configuration.

Our goal is to classify the states---probability measures on $X$---which
are translation invariant (TI) and stationary for the F-SSEP
dynamics, the {\it TIS states}.  Every TIS state is a convex combination
of the extremal TIS (ETIS) states, that is, of the TIS states
which are not proper convex combinations of others, so it suffices to
find the ETIS states.  The ETIS states need not be extremal TI
(ETI)---i.e., ergodic under translations---but since particles are
neither created nor destroyed, extremality in the class of TIS states
suffices to guarantee (see \clem{densities}) that each ETIS state will be
supported on the set $X_\rho$ of configurations having particle density
$\rho$ for some $\rho$ with $0\le\rho\le1$, that is, satisfying
 \be\label{rhodef}
\lim_{N\to\infty}\frac1N\sum_{i=1}^N\eta_i
 = \lim_{N\to\infty}\frac1N\sum_{i=-N}^{-1}\eta_i = \rho.
 \ee
 We will say that a TI state {\it has density $\rho$} if it is supported
on $X_\rho$.  (Note that this condition implies that for each $i\in\bbz$
the expected value of $\eta(i)$ is $\rho$, but is in fact a stronger
statement.)

It is convenient to consider also a second particle system on $\bbz$,
again evolving in discrete time: the {\it symmetric stack model} (SSM).
In this model there are no restrictions on the number of particles at any
site, so that the configuration space is $\Xh=\bbz_+^\bbz$, where
$\bbz_+=\{0,1,2,\ldots\}$.  We denote stack configurations by boldface
letters, and to distinguish explicit stack configurations from F-SSEP
particle configurations we will use italics for the former, so that for
$\n\in \Xh$ we might have $\n(0)=\2$.  To specify the evolution of the
SSM, let us say that the stack at site $i$ is {\it short\/} if
$\n(i)\le\1$ and {\it tall\/} otherwise.  Then in the transition from
$\n_t$ to $\n_{t+1}$ either zero or one particle moves along each bond
$\langle k,k+1\rangle$, either to the left or to the right: if the stacks
at $k$ and $k+1$ are both short then no particle moves on the bond; if
one is short and one tall then a particle moves from the tall to the
short stack, and if both are tall then a particle moves from one to the
other in a randomly chosen direction, with each direction equally likely.

For the SSM we again speak of TIS and ETIS states and, for
$\rhohat<\infty$, of states of density $\rhohat$, where the latter are
those supported on $\Xh_{\rhohat}$, the set of SSM configurations of
density $\rhohat$, defined in parallel with \eqref{rhodef}.  If the
expected value of $\n(0)$ is finite in a TI state on $\Xh$ then we say
that the state is {\it regular}.  A TI state on $X$ is called regular if
it gives zero probability to the configuration $\eta$ for which
$\eta(i)=1$ for all $i$.

The SSM is connected with the F-SSEP through a {\it substitution map}
$\phi:\Xh\to X$: if $\n\in \Xh$ then $\phi(\n)$ is obtained by replacing
each $\n(i)$ with a zero followed by $\n(i)$ ones.  (Such a mapping has
also been used to relate exclusion and zero range processes; see, e.g.,
\cite{EH,FPV}.)  In \capp{subs} we show that this substitution, and
members of a large class of similar substitutions, give rise to a
bijection of the regular TI or ETI states of the models related by the
substitution.  Moreover, we show there that, for the particular
substitution $\phi$ above, the bijection $\Phi_\phi$ of the regular TI
states is also a bijection from the regular TIS or ETIS states of the SSM
to those of the F-SSEP.  Thus for the question of interest here---the
nature and classification of the TIS states---the F-SSEP and SSM are
essentially equivalent, and we can and will pass freely from one to the
other.  Note that if $\n$ has density $\rhohat$ then $\phi(\n)$ has
density $\rho=\rhohat/(1+\rhohat)$; correspondingly, $\Phi_\phi$ carries
regular SSM states of density $\rhohat$ to regular F-SSEP states of
density $\rho$.

Recall now that each regular ETIS state is associated with some density
$\rho$ in the F-SSEP or equivalently $\rhohat=\rho/(1-\rho)$ in the SSM.
The classification of the ETIS states of the models is qualitatively
different in the three density regions $0\le\rho\le1/2$, $\rho=1/2$, and
$1/2\le\rho<1$, or equivalently $0\le\rhohat\le1$, $\rhohat=1$, and
$1\le\rhohat<\infty$.  The states in the first two of these regions are of
course all regular.

Consider first the low density region.  The set of ETIS states of the SSM
with $0\le\rhohat\le1$ is precisely the set of ETI states supported on
$\Fh\subset\Xh$, the set of {\it frozen} SSM configurations for which
every stack has height zero or one (note that in fact $\Fh=X$).  For the
F-SSEP the corresponding result is that the ETIS states are the ETI
states supported on the set $F$ of frozen F-SSEP configurations: those in
which no two adjacent sites are occupied and hence no particle jumps are
possible.

 When $\rhohat=1$ in the SSM ($\rho=1/2$ in the F-SSEP) there are two
families of ETIS states in each model; these describe patterns moving to
the left or to the right, respectively, with speed 1 in the SSM and speed
2 in the F-SSEP.  In the SSM the left-moving family consists of all ETI
states supported on $\Xh_{\rm left}$, the set of configurations in which
no stack has height more than 2, a stack of height 2 can be followed only
by one of height 0, and a stack of height 0 can be preceded only by one
of height 2.  Similarly, the right-moving family consists of the ETI
states on $\Xh_{\rm right}$, the spatial reflection of $\Xh_{\rm left}$.
Two states belong to both families: the state $\muhat^{(1)}$ (which is
also one of the low-density states of the previous paragraph) supported
on the single configuration in which all stacks have height $\1$, and the
state $\muhat^{(2)}$ supported with equal probability on the two
configurations in which stacks of height $\0$ and $\2$ alternate.  The
left- and right-moving families in the F-SSEP are obtained from those of
the SSM via the map $\Phi_\phi$.

All TIS states in the high density region of the SSM, $1\le\rhohat$, are
supported on $\Xha\subset \Xh$, the set of configurations for which no
two adjacent sites both have short stacks.  On $\Xha$ the dynamics
preserves the parity of each stack height, so that if for $\sigma\in X$
we let $\Xha_\sigma\subset\Xha$ be the set of configurations $\n$ for
which $\n(i)$ has parity $(-1)^{\sigma(i)}$ then each $\Xha_\sigma$ is
invariant for the dynamics.  (In these circumstances we call $\sigma$ a
{\it parity sequence}.)

Let $e$ be the parity sequence with $e(i)=0$ for all $i$, so that
$\Xha_e$ is the set of configurations for which each stack height is even
and there are no adjacent zeros.  For each $\rhohat_e\ge1$ we find an
ETIS state $\muhat_e^{(\rhohat_e)}$ on $\Xha_e$ of density $\rhohat_e$,
which we conjecture to be unique: if $\rhohat_e=1$ then
$\muhat_e^{(\rhohat_e)}$ is the state $\muhattwo$ described above, while
if $\rhohat_e>1$ then $\muhat_e^{(\rhohat_e)}$ is a Gibbs state for an
interaction which is simply a one-body potential together with the
constraints---hard-core and evenness---implicit in $\Xha_e$.  Further,
for each such $\rhohat_e$ we obtain from $\muhat_e^{(\rhohat_e)}$ a
family of regular ETIS states on $\Xh^*$, and show that if the conjecture
mentioned above holds then these are all such states.  Specifically, for
each $\rhohat_e\ge1$ and each ETI state $\lambda$ on $X$ there is an ETIS
state $\muhat^{(\rhohat_e,\lambda)}$ for the SSM;
$\muhat^{(\rhohat_e,\lambda)}$ has the distribution of $\eta+\sigma$
(pointwise addition), where $\eta$ has distribution
$\muhat_e^{(\rhohat_e)}$, $\sigma$ has distribution $\lambda$, and $\eta$
and $\sigma$ are independent.  Note that if $\lambda$ has density
$\kappa$ then $\muhat_e^{(\rhohat_e,\lambda)}$ has density
$\rhohat_e+\kappa$.  The corresponding families for the F-SSEP are
obtained via the map $\Phi_\phi$.

\begin{remark}\label{ulmu} A discussion of the TIS states of the
continuous-time version of the model, generalized to include an asymmetry
in the jumps, was given in \cite{AGLS}.  The asymmetry is controlled by a
parameter $p\in[0,1]$: a particle at site $i\in\bbz$ jumps to site $i+1$
(respectively $i-1$) with rate $p$ (resp.~$1-p$), provided that site
$i-1$ (resp.~$i+1$) is occupied and site $i+1$ (resp.~$i-1$) is empty.
For $\rho<1/2$ the TIS states are, as for the current model, just the TI
states supported on $F$, but for the continuous-time model it was
possible to determine the limiting state $\underline\mu$ when the initial
state $\mu_0$ is Bernoulli; rather surprisingly, $\underline\mu$ is
independent of $p$.  ($\underline\mu$ is also \cite{GLS1,GLS2} the
limiting state, with initial state $\mu_0$, under totally asymmetric
discrete-time dynamics.)  For $\rho=1/2$, the unique TIS state is
supported with equal probability on the two configurations in which
occupied and empty sites alternate.  For each $\rho>1/2$ there is again a
unique TIS state, the Gibbs state for a particle system in which the only
interaction is an exclusion rule forbidding adjacent empty sites; the
uniqueness was established via a coupling of the model with the usual
asymmetric simple exclusion process. \end{remark}

\section{Preliminary considerations\label{prelims}}

We here introduce some further notation and provide some simple results
for the F-SSEP and SSM models, often speaking in terms of the F-SSEP with
the understanding that parallel notation will be used, and similar
results hold, for the SSM.  Let us mention several pieces of general
notation: for any sets $A$ and $B$, function $f:A\to B$, and measure
$\lambda$ on $A$ we let $f_*\lambda$ be the measure on $B$ with
$(f_*\lambda)(C)=\lambda(f^{-1}(C))$; moreover, if $B=\bbr$ we let
$\lambda(f)=\int_Af\,d\lambda$ denote the expected value of $f$ under
$\lambda$.  When $C\subset B$ we let $\ind_C:B\to\{0,1\}$ denote the
indicator function of the set $C$.  If $S$ is a finite set then $|S|$
denotes the size of $S$.

Recall from \csect{intro} that the configuration spaces for these models
are $X:=\{0,1\}^\bbz$ and $\Xh:=\bbz_+^\bbz$, respectively, with
$\eta\in X$ and $\n\in\Xh$ denoting configurations.  For $\eta\in X$ and
$j,k\in\bbz$ with $j\le k$ we let
$\eta\rng{j}{k}=(\eta(i))_{j\le i\le k}$ denote the portion of the
configuration $\eta$ lying between sites $j$ and $k$ (inclusive). We will
occasionally use string notation for configurations or partial
configurations, writing for example
$\eta\rng04=\eta(0)\cdots\eta(4)=01101=01^201$.  $\tau:X\to X$ (or
$\tauhat:\Xh\to\Xh$) denotes the translation operator: if $\eta\in X$
then $(\tau\eta)(i)=\eta(i-1)$, if $f$ is any function on $X$ then
$\tau f(\eta)=f(\tau^{-1}\eta)$, and if $\mu$ is a (Borel) measure on $X$
then $\tau$ acts on $\mu$ via $\tau_*$.

It will sometimes be convenient to associate to each F-SSEP configuration
$\eta\in X$ a height profile $h_\eta:\bbz\to\bbz$, which, in the usual
convention, rises by one unit when $\eta(i)=0$ and sinks by one unit when
$\eta(i)=1$.  Specifically,
 $$h_\eta(k)=\begin{cases}\openup4\jot 0,&\hbox{if $k=0$,}\\
\sum_{i=1}^k(-1)^{\eta(i)},& \hbox{if $k>0$,}\\
-\sum_{i=k+1}^{0}(-1)^{\eta(i)},& \hbox{if $k<0$.}\end{cases}\label{hdef}$$
We do not introduce height profiles for SSM configurations.

From the somewhat informal description of the dynamics of the models
given in \csect{intro} it is straightforward but tedious to specify, for
a configuration $\eta\in X$ and a measurable subset $A\subset X$, the
transition kernel $Q(\eta,A)$ of the F-SSEP Markov process, or similarly
the kernel $\Qh(\n,B)$ for the SSM model.  We omit the details.  A
measure $\mu$ on $X$ is stationary if $\mu=\mu Q$; here
$(\mu Q)(A)=\int_XQ(\cdot,A)\,d\mu$ (we also write
$\mu Q^n:=(\mu Q^{n-1})Q$ for $n\ge2$).

In the remainder of the paper we will consider primarily {\it regular}
(see \csect{intro}) states on $X$ and $\Xh$; the sets of regular TI, ETI,
TIS, and ETIS states for the F-SSEP are denoted by $\M(X)$, $\Mb(X)$,
$\M_s(X)$, and $\Mb_s(X)$, respectively.  We write similarly $\M(\Xh)$,
etc., as well as $\M(A)$, $\M(\hat A)$, etc., for $A\subset X$ or
$\hat A\subset\Xh$ TI sets.  As a consequence of the results of
\capp{subs} (see also the discussion of \csect{intro}) we have
immediately:

\begin{theorem}\label{bijection} There exists a bijection
$\Phi_\phi:\M(\Xh)\to\M(X)$, arising from the substitution map $\phi$
defined in \csect{intro}, which satisfies $\Phi_\phi(\Mb(\Xh))=\Mb(X)$,
$\Phi_\phi(\M_s(\Xh))=\M_s(X)$, and $\Phi_\phi(\Mb_s(\Xh))=\Mb_s(X)$.
$\Phi_\phi$ carries states of density $\rhohat$ to states of density
$\rhohat/(1+\rhohat)$.  Moreover, if $\Ah\subset\Xh$ is TI then there is
a similarly defined bijection from $\M(\Ah)$ to $\M(A)$, etc., where $A$
is the minimal TI subset of $X$ containing $\phi(\Ah)$.\end{theorem}

\begin{remark}\label{finite} It is clear that one may also define, in a
straightforward way, models with a fixed number of particles moving on a
finite ring under either the F-SSEP or SSM dynamics.  These models will
play a role in \csects{rhohalf}{stkmeas}.  \end{remark}

Since we are studying stationary states it is natural to introduce the
set of space-time F-SSEP configurations
$X_2=\{0,1\}^{\bbz^2}=\{(\xi_t(i))_{(t,i)\in\bbz^2}\}$; a state
$\mu\in\M(X)$ which is stationary for the dynamics induces a ``path
measure'' on $X_2$, invariant under vertical translation, which we denote
$P_\mu$.  $\Xh_2$ and $\Ph_{\muhat}$ denote the corresponding SSM quantities.

In \csectt{lowrho}{rhobig} we will describe all ETIS states for the two
models.  As indicated in \csect{intro}, these fall into certain natural
groups, which we will call $\lambda$-families.  In this context we write
$\L$ for the set of ergodic TI measures on $X$ (in fact, $\L=\Mb(X)$, but
the special role that this space plays here motivates a special symbol).  

\begin{definition}\label{lamfam} A {\it$\lambda$-family} is a collection
of ETIS states, for either the SSM or the F-SSEP, which is bijectively
equivalent (with a ``natural'' bijection) to $\L$.  We think of $\L$ as
indexing the $\lambda$-family and let $\lambda\in\L$ denote a typical
index.  We will typically write $\Fch_*$ and $\Fc_*$ for
$\lambda$-families for the SSM and F-SSEP, respectively, with $*$ a
subscript distinguishing the various families and with
$\Fc_*=\Phi_\phi(\Fch_*)$.  $\Psih_*:\L\to\Fch_*$ and
$\Psi_*=\Phi_\phi\circ\Psih_*:\L\to\Fc_*$ are the corresponding indexing
bijections.\end{definition}

Certain simple spatially-periodic configurations and related states,
some already mentioned in \csect{intro}, will play a special role in our
discussions.  Let $\n^{(1)},\n^{(2)}\in\Xh$ be the configurations with
$\n^{(1)}(i)=\1$ and $\n^{(2)}(i)=(\1+(-\1)^i)$ for all $i$; in
\csect{intro} we introduced the states
$\muhat^{(1)},\muhat^{(2)}\in\M(\Xh)$ defined by
$\muhat^{(1)}=\delta_{\n^{(1)}}$,
$\muhat^{(2)}=(\delta_{\n^{(2)}}+\delta_{\tau\n^{(2)}})/2$.  The
corresponding states $\mu^{(1)}=\Phi_\phi(\mu^{(1)})$,
$\mu^{(2)}=\Phi_\phi(\mu^{(2)})$ are given by
 \be\label{mu12}
 \mu^{(1)}=\frac12\sum_{j=0}^1\delta_{\tau^j\eta^{(1)}}
  \qquad\text{and}\qquad
  \mu^{(2)}=\frac14\sum_{j=0}^3\delta_{\tau^j\eta^{(2)}},
 \ee
 where $\eta^{(1)}\in X$ is the period-two configuration with
$\eta^{(1)}(1{:}2)=10$ and $\eta^{(2)}\in X$ is the period-four
configuration with $\eta^{(2)}(1{:}4)=1100$.  It is easy to check
directly that $\muhat^{(1)}$ and $\muhat^{(2)}$ are ETIS states for the
SSM, as are $\mu^{(1)}$ and $\mu^{(2)}$ for the F-SSEP.

We conclude this section with three general results; we state these for the
F-SSEP, but the obvious translations to the SSM also hold.  Recall from
\csect{intro} that we say that a state $\mu\in\M(X)$ has density $\rho$
if it is supported on the space $X_\rho$ (see \eqref{rhodef}).

\begin{lemma}\label{densities} Every ETIS state
$\mu\in\Mb_s(X)$ has a definite density $\rho$ and satisfies either
$\mu(F)=0$ or $\mu(F)=1$.  \end{lemma}

\begin{proof} Take $\mu\in \M_s(X)$; it suffices to show that $\mu$ is a
convex combination of states in $\M_s(X)$ with a definite density $\rho$
and for which $F$ has probability 0 or 1. Let $\nu=r_*\mu$; here
$r:X\to\bbr$,
$r(\eta):=\lim_{N\to\infty}(2N+1)^{-1}\sum_{i=-N}^N\eta(i)$, is defined
$\mu$-a.e.~by the ergodic theorem.  $\nu$ is just the distribution of the
density with respect to $\mu$.  Then \cite{Kallenberg} there exists a
unique regular conditional probability distribution
$(\mu_\rho)_{\rho\in[0,1]}$ for $\mu$ such that $\mu_\rho$ has density
$\rho$ and for any measurable $A\subset X$,
 \be
\mu(A)=\int_{0\le \rho\le 1}\mu_\rho(A)d\nu(\rho).
 \ee
 Since $F$ is invariant under the dynamics, i.e., $Q(\eta,F)=1$ for
$\eta\in F$, we see that if we further write
$\mu_\rho=\mu_\rho\big|_F+\mu_\rho\big|_{X\setminus F}$ we obtain, after
normalization of $\mu_\rho\big|_F$ and $\mu_\rho\big|_{X\setminus F}$,
the desired representation.  More details are given in \cite{AGLS}.
\end{proof}

 We next give a lemma which shows that a connected portion of an F-SSEP
configuration, other than possibly a single 1 at some site, cannot be
frozen unless the entire configuration is.

\begin{lemma}\label{frpart} Let $\mu\in\M_s(X)$ satisfy $\mu(F)=0$.
Suppose that $I\subset\bbz$ is an interval, that $\theta\in\{0,1\}^I$,
and that either $|I|\ge2$ or $I=\{i\}$ and $\theta(i)=0$.  Then
$P_\mu(\{\xi\in X_2\mid\text{for all $t$,
}\xi_t\big|_I=\theta\})=0$.  \end{lemma}

We remark that the possibility $I=\bbz$, $\theta(i)=1$ for all $i$, an
apparent counterexample to \clem{frpart}, is forbidden by the regularity
of $\mu$.

\begin{proofof}{\clem{frpart}} Let $A,B,C\subset X_2$ denote the sets of
space-time histories $\xi$ such that, respectively, $\xi_t(0)=0$ for all
$t$ (we write $\xi_t(0)\equiv0$), $\xi_t\rng{-1}0\equiv01$, and
$\xi_t\rng{-1}0\equiv11$, and such that in each case $\xi_t(1)$ changes
infinitely often as $t\nearrow\infty$.  We show that
$P_\mu(A)=P_\mu(B)=P_\mu(C)=0$; using the translation invariance of
$\mu$, translation invariance in time of $P_\mu$, and reflection
invariance of the system (although the latter is not really needed), one
sees easily that this implies the result.

 First, we observe that if $\xi\in X_2$ satisfies $\xi_t(0)\equiv0$ and
$s$ is such that $\xi_s(1)=1$ then, $P_\mu$-a.s., $\xi_t(1)=1$
for all $t\ge s$ and so $\xi\notin A$.  Thus $P_\mu(A)=0$.  Next, suppose
that $\xi_t\rng{-1}0\equiv01$; if $\xi_s(1)=1$ for some $s$ then,
$P_\mu$-a.s., $\xi_{s+1}(-1)=1$, a contradiction. Thus
$P_\mu(\xi\in B)=0$.  Finally, if $\xi\in C$ we define
$(\tau_k)_{k\in\bbz_+}$, with $0\le\tau_0<\tau_1<\cdots$, to be the
nonnegative times satisfying $\xi_{\tau_k}(1)=0$; the $\tau_k$ are well
defined on $C$.  Now $\xi\in C$ implies that $\xi_{\tau_k+1}(0)=1$, and
this is possible only if $\xi_{\tau_k}\rng{-1}3=11011$ and, moreover,
$L_k$ occurs, where $L_k$ is the event that when two particles attempt at
time $\tau_k$ to jump to site $1$, it is the leftward jump which
succeeds.  Given $C$, the $L_k$ are independent events, each with
probability 1/2, so that $P_\mu(C)=0$.\end{proofof}

Finally, we give a result showing that non-frozen TIS states cannot have
too large a local density of 0's.
 
\begin{lemma}\label{noincrease} Suppose that $\mu\in\M_s(X)$ satisfies
$\mu(F)=0$.  Then:
 \par\smallskip\noindent
 (a)  $\mu$-a.s., no configuration contains three consecutive 0's.
 \par\smallskip\noindent
 (b) $\mu$-a.s., the height profile $h_\eta$ of the configuration
$\eta$ does not increase by more than two units over any interval.
\end{lemma}

\begin{proof} (a) Elementary analysis of the dynamics shows that,
$P_\mu$-a.s., if $\xi\in X_2$ satisfies
$\xi_t(i)=\xi_t(i+1)=\xi_t(i+2)=0$ for some $t,i\in\bbz$, then also
$\xi_{t-1}(i)=\xi_{t-1}(i+1)=\xi_{t-1}(i+2)=0$.  But then
$\xi_s(i)=\xi_s(i+1)=\xi_s(i+2)=0$ for all $s\le t$ and, by the
invariance of $P_\mu$ under time translations, for all $s$.  The
conclusion follows from \clem{frpart}.

 \smallskip\noindent (b) Suppose that the conclusion is false.  By
translation invariance we may assume that for some (necessarily odd)
positive integer $j$, $\mu(A_j)>0$, where $A_j=\{\eta\mid h_\eta(j)=3\}$.
We take $j$ to be the minimal value for which this holds.  By (a), $j>3$,
and minimality of $j$ implies that if $\eta\in A_j$ then
a.s.~$\eta\rng1j=00(10)^{(j-3)/2}0$.  Now by \clem{frpart} (as in (a)),
$P_\mu(B)>0$, where
$B=\{\xi\in X_2\mid \xi_0\in A_j,\xi_{-1}\notin A_j\}$.  On the other
hand, if $\xi\in B$ then elementary analysis shows that, $P_\mu$-a.s.,
some translate of $\xi_{-1}$ belongs to $A_{j'}$ for some $j'<j$ (for
example, if $j=9$ then $\xi_{-1}\rng1j$ must be one of $1100(10)^20$,
$00(10)^2011$, $110010011$, or $000111000$).  This contradicts the
minimality of $j$.  \end{proof}

\section{Low density\label{lowrho}}

In this section we describe all TIS states of low density ($\rho<1/2$ for
the F-SSEP, $\rhohat<1$ for the SSM).  We first show that all ETIS states
for the F-SSEP are frozen.

\begin{lemma}\label{smallrho} Let $\mu\in\Mb_s(X)$ have density
$\rho<1/2$.  Then $\mu(F)=1$.  \end{lemma}

\begin{proof} By \clem{densities} it suffices to show that $\mu(F)>0$.
But if $\mu(F)=0$ then for any $k>0$,
 \be
2\rho-1=\frac1k\mu\left(\sum_{i=1}^k(2\eta(i)-1)\right)
  =-\frac1k\mu\bigl(h(k)\bigr)\ge-\frac2k,
 \ee
 by \clem{noincrease}(b). Since $k$ is arbitrary, $2\rho-1\ge0$.
\end{proof}

To describe all low-density states we introduce $\Fc_{\rm low}:=\Mb(F)$,
the set of ETI states supported on $F$, and $\Fch_{\rm low}:=\Mb(\Fh)$.
Since $\Fh=X$, $\Fch_{\rm low}=\Mb(X)=\L$, and thus $\Fch_{\rm low}$ is
indeed a $\lambda$-family (see \cdef{lamfam}), with indexing map
$\Psih_{\rm low}$ the identity.  One checks easily that
$\Phi_\phi(\Fch_{\rm low})=\Fc_{\rm low}$, so that $\Fc_{\rm low}$ is
also a $\lambda$-family, with indexing map $\Psi_{\rm low}=\Phi_\phi$.
Note that $\muhat^{(1)}\in\Fch_{\rm low}$ has density $1$ and
$\mu^{(1)}=\Phi_\phi(\muhat^{(1)})\in\Fc_{\rm low}$ has density 1/2, but
that otherwise the states in $\Fch_{\rm low}$ and $\Fc_{\rm low}$ have
densities less than 1 or less than 1/2, respectively.

\begin{theorem}\label{main3} (a) The set of ETIS states for the F-SSEP
with density $\rho<1/2$ is $\Fc_{\rm low}\setminus\{\mu^{(1)}\}$, and
(b)~the set of ETIS states for the SSM with density $\rhohat<1$ is
$\Fch_{\rm low}\setminus\{\muhat^{(1)}\}$.  \end{theorem}

\begin{proof}For (a), note that the inclusion
$\Fc_{\rm low}\subset\Mb_s(X)$ is trivial;, while conversely every
state in $\Mb_s(X)$ with density less than 1/2 belongs to
$\Fc_{\rm low}$ by \clem{smallrho}.  By virtue of \cthm{bijection}, (b)
is an immediate consequence of (a). \end{proof}

Finally we ask the question: if the F-SSEP is started in an initial state
$\mu_0$ which is a Bernoulli measure with density $\rho<1/2$, what is the
final distribution of frozen configurations?  As discussed in
\crem{ulmu}, in earlier work we answered this question for several
other facilitated exclusion processes, finding a common limiting
distribution $\underline\mu$.  For the current model the limit is
different (see below), but beyond that we have only a partial
description.

Let $\mu_t=\mu_{t-1}Q$, $t=1,2,\ldots$, be the state at time $t$, with
$\mu_0$ as above, and for $\eta\in X$ let
$S_\eta=\{i\in\bbz\mid\eta\rng{i-2}{i}=000\}$.

\begin{theorem}\label{renewal}The limiting measure
$\mu_\infty=\lim_{t\to\infty}\mu_t$ exists and satisfies
$\mu_\infty(\{0\in S_\eta\})>0$. Under the conditional measure
$\mu_\infty(\cdot\mid 0\in S_\eta)$, $S$ is a renewal process.
\end{theorem}

\begin{proof}Let $\Pt$ be the path measure on
$\Xt_2=\{0,1\}^{\bbz_+\times\bbz}$ obtained from the initial state
$\mu_0$ and the F-SSEP dynamics.  Elementary analysis shows that if
$\xi\in X_2$ satisfies $\xi_t\rng {i-2}i=000$ then $\Pt$-a.s.~also
$\xi_s\rng {i-2}i=000$ for $0\le s\le t$.  But then for all $i\in \bbz$,
 \begin{align}\nonumber
 B_i&:=\{\xi\mid\xi_t\rng {i-2}i=000\text{ for all $t\ge0$}\}\\
 &= \{\xi\mid\xi_t\rng {i-2}i=000\text{ for all 
 sufficiently large $t$}\}. \label{th8}
 \end{align}
 Moreover, $\Pt(B_i) =q^2(1-\rho)^3$, with $q$ the probability that, for
an initial measure under which the configuration on
$\bbn:=\{i\in\bbz\mid i>0\}$ is distributed as a Bernoulli measure with
density $\rho$ but all other sites are empty, no particle crosses the
$\langle0,1\rangle$ bond at any time during the evolution.

Next we show that $q>0$.  Let $A_t$ be the event that for all $L\ge1$
there are in the configuration $\xi_t$ at most $L/2$ particles on sites
$1,2,\ldots,L$. A standard gambler's ruin computation \cite{GS} shows
that $\mu_0(A_0)=(1-2\rho)/(1-\rho)$.  Now when $A_t$ holds no particle
can cross the bond $\langle0,1\rangle$ at the next time step, from the
$L=1$ condition. Moreover, $A_t\subset A_{t+1}$.  Thus
$q\ge\mu_0(A_0)>0$.

 To establish the existence of $\mu_\infty$, and in fact a stronger
result, the $\Pt$-almost sure existence of
$\eta_\infty:=\lim_{t\to\infty}\xi_t$, one shows from $\Pt(B_i)>0$ that
for any $L>0$, $\Pt\bigl(\bigcup_{k,l\ge L}B_{-k}\cap B_l\bigr)=1$.  For
$\xi\in B_{-k}\cap B_l$, simple considerations of the system in a finite
region then imply that $\lim_{t\to\infty}\xi_t\big|_{[-k+1,l-3]}$ exists,
$\Pt$-a.s.  For more details see the proof of Lemma~3.6 of \cite{AGLS}.

Since $\mu_\infty$ is the distribution of $\eta_\infty$ under $\Pt$ we
must show that $S_{\eta_\infty}$ is a renewal process under the
conditional measure $\Pt(\cdot\mid0\in S_{\eta_\infty})$.  Now from
\eqref{th8} we have that $i\in S_{\eta_\infty}$ iff $\xi\in B_i$, and if
we condition on $0\in S_{\eta_\infty}$, that is, on $B_0$, then what
happens to the left of site $-2$ is independent of what happens to the
right of site $0$.  Thus if, under this conditioning, we label the points
of $S_{\eta_\infty}$ sequentially as $(s_k)_{k\in\bbz}$, with $s_0=0$, so
that $\xi_t\rng{s_k-2}{s_k}=000$ for all $t$ and $k$, the differences
$s_k-s_{k-1}$ are independent.  \end{proof}

Let us condition on $0\in S_{\eta_\infty}$ and adopt the notation of the
previous proof.  Then either $s_1=1$, an event with (conditional)
probability $1-\rho$, or $s_1\ge4$; in the latter case
$\eta_\infty\rng{-2}{s_1}=000\,\sigma\,000$, with $\sigma$ any string
which begins and ends with 1 and contains no substrings $11$ or $000$.
To complete the description of $\mu_\infty$ one would need to find the
distribution of these $\sigma$ (and hence of $s_1-s_0$); we have only
partial results in this direction.  However, one finds easily that, for
example, $\mu_\infty(\eta\rng16=101000\mid0\in S_\eta)=\rho^2(1-\rho)^4$;
on the other hand, with $\underline\mu$ as introduced in \crem{ulmu},
$\underline\mu(\eta\rng16=101000\mid0\in S_\eta)=2\rho^2(1-\rho)^4$
\cite{AGLS,GLS1,GLS2}, establishing the difference of $\mu_\infty$ and
$\underline\mu$.

\section{Densities $\rho= 1/2$ and $\rhohat=1$\label{rhohalf}}

In this section we describe all ETIS states of density 1/2 for the F-SSEP
or density 1 for the SSM.  We first show that, for such a state in the
F-SSEP, the height profile $h_\eta$ (see \eqref{hdef}) is a.s.~confined
to a strip of height at most two.  For $\eta\in X$ we let
$\Delta(\eta)=\sup_{i<j}|h_\eta(j)-h_\eta(i)|$.

\begin{lemma}\label{strip} If $\mu\in\Mb_s(X)$ has density $\rho=1/2$ then,
$\mu$-a.s., $\Delta(\eta)\le2$.  \end{lemma}

\begin{proof} By \clem{densities} we may assume that either $\mu(F)=0$ or
$\mu(F)=1$. Now $\mu^{(1)}$ is the only TI state of density 1/2 supported on
$F$, and it satisfies the conclusion of the lemma.  Consider then the case
$\mu(F)=0$.  By \clem{noincrease}(b), $h_\eta(j)-h_\eta(i)\le2$ for
$i<j$, so by translation invariance it suffices to show that for any
$j>0$, $h_\eta(j)\ge-2$ $\mu$-a.s.; moreover, it suffices to
verify the result for each ergodic component of $\mu$.

Suppose then that $\mu$ is ergodic but that for some $j>0$, which we take
to be minimal, $\mu\bigl(\{\eta\mid h_\eta(j)\le-3\}\bigr)>0$.
Then from \clem{noincrease}(b) and the minimality of $j$, also 
$\mu\bigl(\{\eta\mid\eta(0)>\sup_{i>0}h_\eta(i)\}\bigr)>0$, and with
ergodicity this implies that, $\mu$-a.s., $h_\eta$ has a negative
overall slope.  But this is inconsistent with $\mu$ having density $1/2$.
\end{proof}

We can now classify the states in $\Mb_s(X)$ with density 1/2.
Consider a general $\eta$ with $\Delta(\eta)\le2$; a portion of a
typical height profile (with $\Delta(\eta)=2$) is shown in
\cfig{profhalf}.  As indicated there, we may partition $\eta$ into {\it
left-moving (L)}, {\it right-moving (R)}, and {\it transition (T)}
regions, where T denotes a maximal region in which pairs of 0's alternate
with pairs of 1's, L a region between two T regions having the form
$(10)^k$ for some $k\ge1$, and R a region between two T regions having
the form $(01)^k$, $k\ge1$.  

\begin{figure}[ht]
  \centering\bigskip
  \hbox to \hsize{\hfill\includegraphics{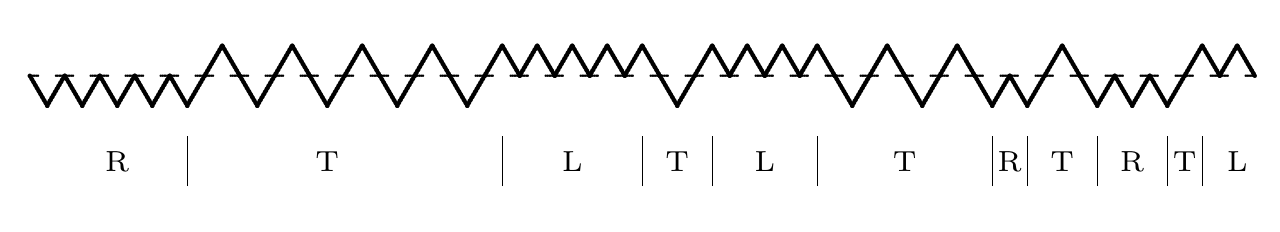}\hfill}
\caption{Portion of a typical height profile, showing Left-moving,
Right-moving, and Transition regions.} \label{profhalf} \end{figure}

 The naming of these regions reflects the fact that a configuration in
which only $L$ and $T$ (respectively $R$ and $T$) regions appear
translates to the left (respectively right) at velocity 2 under the
dynamics: if $\eta_t$ is such a configuration, then
$\eta_{t+1}=\tau^{-2}\eta_t$ (respectively $\eta_{t+1}=\tau^2\eta$).  We
will see below that all TIS states are supported on such configurations.
The two simplest examples of such states, $\mu^{(1)}$ and $\mu^{(2)}$,
were defined in \eqref{mu12}; $\mu^{(1)}$ may be viewed as supported on
configurations with a single L, or equivalently a single R, region, and
$\mu^{(2)}$ is supported on configurations consisting of a single T
region.

  It is straightforward to work out further rules for the evolution of
the configurations with $\Delta(\eta)\le2$.  Such configurations cannot
contain the pattern $11011$, and hence the evolution is deterministic.
Boundaries between regions usually move with velocity 2, L$|$T and T$|$L
boundaries to the left and R$|$T and T$|$R boundaries to the right.
However, if R and L regions are separated by a T region which is a single
pair of 0's, R$|$00$|$L, then both boundaries are stationary.  If in the
resulting T$|$R$|$00$|$L$|$T situation the $L$ region is shorter than the
$R$ region then the (left-moving) L$|$T boundary will eventually reach
the 00$|$L boundary, at which time the L region will disappear and the 00
and T regions amalgamate into a single T region; the situation is similar
when the R region is shorter, or the two regions are the same length.

Following these ideas one easily sees that, on a ring, any initial
configuration evolves to one in which either no L region, or no R region,
occurs.  No such conclusion is possible when the ring is replaced by
$\bbz$, but similar considerations do imply the following: For any
configuration $\eta$ with $\Delta(\eta)\leq2$, the minimum size
$\delta=\delta(\eta)$ of a maximal uniform block---a maximal right block,
a region of the form $T|R|T|\cdots|T|R|T$ preceded and followed by $L$
regions, or a maximal left block, defined in a parallel way---must
increase, if $\delta < \infty$, within a time proportional to $\delta$.
From this it follows that any TI stationary state for $\rho = 1/2$ must
be supported on configurations with $\delta = \infty$, i.e., on
configurations consisting of a single maximal uniform block or of a
maximal right block adjacent to a maximal left block. However, the set of
configurations of the latter sort must have probability 0 in any TI
state.

Rather than filling in the straightforward details of the argument just
sketched, we now give a different and shorter (though perhaps less
intuitive) argument.  Let $X_{\rm left}$ (respectively $X_{\rm right}$)
denote the set of configurations $\eta\in X$ with $\Delta(\eta)=2$ which
contain no R (respectively no L) region.  By convention we suppose that
$X_{\rm left}$ and $X_{\rm right}$ contain $\eta^{(1)}$ and
$\tau\eta^{(1)}$ (see \csect{prelims}); then
$X_{\rm left}\cap X_{\rm right}$ consists of $\eta^{(1)}$, $\eta^{(2)}$,
and their translates.  Recall also the spaces $X_{\rho}$, $\Xh_{\rhohat}$,
$\Xh_{\rm left}$ and $\Xh_{\rm right}$, defined in \csect{intro}.

\begin{theorem}\label{leftright} (a) When $\rho=1/2$,
$\Mb_s(X_\rho)=\Mb(X_{\rm left})\cup\Mb(X_{\rm right})$.
\par\smallskip\noindent
(b) When $\rhohat=1$,
$\Mb_s(\Xh_{\rhohat})=\Mb(\Xh_{\rm left})\cup\Mb(\Xh_{\rm right})$.
\end{theorem}

Before giving the proof we explain how the ETIS states of the theorem are
organized into the $\lambda$-families of \cdef{lamfam}.  Consider first
$\Mb_s(\Xh)=\Mb(\Xh_{\rm left})\cup\Mb(\Xh_{\rm right})$; here
$\Mb(\Xh_{\rm left})$ and $\Mb(\Xh_{\rm right})$ are $\lambda$-families
which we denote respectively $\Fch_{\rm left}$ and $\Fch_{\rm right}$.
The substitution map $\phileft:X\to \Xh$ given by $1\to\2\0$, $0\to\1$
gives rise to an indexing bijection
$\Psih_{\rm left}=\Phi_{\phileft}:\L\to\Fch_{\rm left}$, and similarly we
obtain $\Psih_{\rm right}:\L\to\Fch_{\rm right}$ from $\phiright$, the
substitution map sending $1\to\0\2$, $0\to\1$.  The indexing maps for
$\Fc_{\rm left}=\Mb(X_{\rm left})$ and
$\Fc_{\rm right}=\Mb(X_{\rm right})$ are
$\Psi_{\rm left}=\Phi_\phi\circ\Psih_{\rm left}$ and
$\Psi_{\rm right}=\Phi_\phi\circ\Psih_{\rm right}$, or can be obtained
directly from the substitution maps $1\to1100$, $0\to10$ and $1\to0011$,
$0\to01$, respectively.  Note that
$\Fch_{\rm left}\cap\Fch_{\rm right}=\{\muhat^{(1)},\muhat^{(2)}\}$ and
$\Fc_{\rm left}\cap\Fc_{\rm right}=\{\mu^{(1)},\mu^{(2)}\}$.

\begin{proofof}{\cthm{leftright}} (a) Since the F-SSEP dynamics on
$X_{\rm left}$ or $X_{\rm right}$ is simply left or right translation,
respectively, clearly
$\Mb(X_{\rm left})\cup\Mb(X_{\rm right})\subset\Mb_s(X)$.  Suppose
conversely that $\mu\in\Mb_s(X)$.  Since $\mu$ is an ETIS state
and $X_{\rm left}$ and $X_{\rm right}$ are invariant under the dynamics,
it suffices to prove that $\mu(X_{\rm left}\cup X_{\rm right})=1$, i.e.,
that the set of configurations $\eta$ which have $\Delta(\eta)=2$, and
contain both L and T regions, has measure zero.

We first show that,
$\mu$-a.s., no configuration contains the string $010010$
(corresponding to regions R$|$T$|$L with T having just two sites).  For
let $E$ be the event that sites 1 and 2 both belong to an R region and
that $\eta(1{:}2)=01$.  Then from the dynamical rules,
 \begin{align}
 P_\mu(\xi_{t+1}\in E)&=P_\mu(\xi_t\in E)-P_\mu(\xi_t(-1{:}2)=1101)
 +P_\mu(\xi_t(-1{:}4)=010011)\nonumber\\
 &=P_\mu(\xi_t\in E)-P_\mu(\xi_t(-1{:}2)=1101)
     +P_\mu(\xi_t(-1{:}2)=0100)\nonumber\\
   &\hskip40pt-P_\mu(\xi_t(-1{:}4)=010010).\label{xx}
 \end{align}
  But  $\mu\bigl(\eta(-1{:}2)=1101)=\mu\bigl(\eta(-1{:}2)=0100\bigr)$,
since the density of left and right ends of R regions must be the same,
so that, with the stationarity of $\mu$, \eqref{xx} implies that
$\mu\bigl(\eta(-1{:}4)=010010\bigr)=0$, as claimed.

 Now if $\mu(X_{\rm left}\cup X_{\rm right})<1$ then there is a minimal $k$
such that $\mu\bigl(\eta(1{:}4k+6)=0100(1100)^k10\bigr)>0$.  From the
claim above, $k>0$.  But if $\xi_t(1{:}4k+6)=0100(1100)^k10\bigr)$ then
$\xi_{t+1}(3{:}4k+4)=0100(1100)^{k-1}10\bigr)$, contradicting the
minimality of $k$.  

 \smallskip\noindent
 (b) By (a) and \cthm{bijection} (using $\Ah=\Xh_{\rm left}$ and
$\Ah=\Xh_{\rm right}$) we have
 \begin{align}\nonumber
 \Mb_s(\Xh_1)=\Phi_\phi^{-1}\bigl(\Mb_s(X_{1/2})\bigr)
 &=\Phi_\phi^{-1}\bigl(\Mb(X_{\rm left})\bigr)\cup
  \Phi_\phi^{-1}\bigl(\Mb(X_{\rm right})\bigr)\\
 &= \Mb(\Xh_{\rm left})\cup\Mb(\Xh_{\rm right}).
 \end{align}
\end{proofof}

 \section{High density\label{rhobig}}

Most of our discussion of the high density region, $\rho>1/2$ for the
F-SSEP or $\rhohat>1$ for the SSM, will be carried our for the SSM.  We
first obtain a reduction of the configuration space of the model to the
space $\Xha\subset\Xh$, the set of configurations for which no two
adjacent sites both have short stacks (see \csect{intro}).  Note that any
extremal TI state in $\Xha_e$ must have density $\rhohat_e$ satisfying
$\rhohat_e\ge1$.  For future reference we also note that $H:=\phi(\Xha)$
is the set of F-SSEP configurations which do not contain any of the
substrings 000, 0100, 0010, and 01010.

 \begin{theorem}\label{tricky}Every TIS state for the stack dynamics
with density $\rhohat>1$ is supported on $\Xha$; equivalently, every
TIS state for the F-SSEP dynamics with density $\rho>1/2$ is
supported on $H$.  \end{theorem}

 \begin{proof}It is convenient to work primarily in the stack model.  We
observe first that, since $\0\0$ in the stack model corresponds to $000$
in the F-SSEP model, \clem{noincrease} implies that every
$\muhat\in\M_s(\Xh)$ with density $\hat\rho>1$ assigns
zero probability to the set of all configurations containing the
substring $\0\0$.

Now we suppose that $\muhat\in\M_s(\Xh)$ has density $\hat\rho>1$ and is
such that configurations containing the string $\1\0$ occur with nonzero
probability under $\muhat$, and derive a contradiction; we may assume
that $\muhat$ is ergodic.  For $(t,i)\in\bbz^2$ let
$E_{t,i}\subset \Xh_2$ be the event that $\m\in \Xh_2$ satisfies
$\m_t\rng{i}{i+1})=\1\0$; then for $\m\in E_{t,i}$ we know from
\clem{frpart} that there must a.s.~exist times $t_0$ and $t_1$, with
$t_0\le t<t_1$, such that $\m\in E_{s,i}$ for $t_0\le s<t_1$ but
$\m\notin E_{t_0-1,i}$ and $\m\notin E_{t_1,i}$.  Then elementary
consideration of the dynamics, using the fact the $\0\0$ does not occur,
shows that necessarily $\m_{t_0-1}\rng{i-1}{i}=\1\0$, i.e., the $\1\0$
string cannot be ``created'' at sites $(i,i+1)$ but rather ``moves''
there from sites $(i-1,i)$.  But then such a string cannot vanish, since
$\muhat$ is stationary and hence the density of $\1\0$ strings is
constant in time, and again simple considerations show that necessarily
$m_{t_1-1}\rng{i}{i+2})=\1\0\2$ and $m_{t_1}\rng{i+1}{i+2}=\1\0$.  The
conclusion is that, $\Ph_{\muhat}$-a.s., $\1\0$ substrings persist
throughout time, moving to the right in the sense that there exist times
$t_0<t_1<t_2<\cdots$ such that $m_t\rng{i+k}{i+k+1}=\1\0$ for
$t_k\le t<t_{k+1}$.

A similar analysis applies to $\0\1$ strings, except that these move to
the left.  But we claim that $\muhat$ cannot give positive density to both
$\1\0$ and $\0\1$ strings.  For if it did, both would occur with positive
frequency in the same configuration, by the ergodicity of $\muhat$, and
would then ``collide'', that is, the configuration $\1\0\1$ would occur
with positive frequency ($\1\0\0\1$ is ruled out because $\0\0$ is), and
then neither could continue to move without the destruction of the other,
a contradiction.  Thus we may assume, without loss of generality, that
$\muhat$ supports a positive density of $\1\0$ substrings but that,
$\muhat$-a.s., $\0\1$ strings do not occur.

Now let $k$ be the minimal positive integer such that $\muhat(A_k)>0$, where
$A_k:=\{\n\in \Xh\mid\n\rng12=\1\0\hbox{ and }\sum_{i=1}^k\n(i)>k\}$; such
a $k$ exists because $\hat\rho>1$ and we have assumed that
$\muhat(\n\rng12)=\1\0)>0$.  From the fact that $\0\0$ and $\0\1$ do not
occur it follows that if $\n_t\in A_k$ then $\n_t\rng1k$ has one of two
forms: (i)~$\n_t\rng1k=\1\0\2(\0\2)^p\1^qx$, with $x\ge2$, or
$(ii)~\n_t\rng1k=\1\0(\2\0)^py$, with $y\ge3$; here $p$ and $q$ are
nonnegative integers.  We show that each of these cases leads to a
contradiction.

If (i) holds then one checks, by considering various special cases, that
with probability at least $1/2$ either (i.a)~neither $\n_{t+1}\rng12=\1\0$
nor $\n_{t+1}\rng23=\1\0$, contradicting the possible evolution of a
$\1\0$ pair as discussed above, or (i.b)~$\tau^{-1}\n_{t+1}$ lies in case
(i) of $A_{k-2}$, contradicting the minimality of $k$.  For example, if
$p=q=0$, so that $k=4$ and $\n_t\rng14=\1\0\2x$, then $\n_t(2:3)=\1\2$
with probability $1/2$ (case 1.a)), while if $q\ge2$ then
$n_{t+1}\rng2{k-2}=\1\0\2(\0\2)^p\1^{q-2}2$, so that
$\tau^{-1}\n_{t+1}\in A_{k-2}$ (case (1.b)).

Suppose then that (ii) holds.  If $p=0$ then, as in case (i.a), the
$\1\0$ disappears, and if $p\ge1$ and $y\ge4$ then
$\tau^{-1}\n_{t+1}\in A_{k-1}$; each of these outcomes contradicts our
assumptions.  If $p\ge1$ and $y=3$ then we must consider an additional
site, and so write $\n_t(k+1)=z$.  If $z\ge2$
then, with probability 1/2, $\n_{t+1}\rng2k=\1\0\2(\0\2)^{p-1}\3$ and so
$\tau^{-1}\n_{t+1}\in A_{k-1}$.  Finally, if $z=1$ or $z=0$ then
$\n_{t+1}\rng2k=\1\0\2(\0\2)^{p-1}\1z'$ with $z'\ge2$, so that 
$\tau^{-1}\n_{t+1}$ lies in $A_k$ but falls under case (i), and this is a
contradiction by the analysis of the previous paragraph.  For this last
step we need to observe that if $z=0$ then necessarily $\n_t(k+2)\ge2$,
since neither $\0\0$ nor $\0\1$ can occur in $\n_t$. 

To complete the proof we must show that, $\muhat$-a.s., the string
$\1\1$ cannot occur in $\n$.  But, as for the initial analysis of $\1\0$
above, if $\m\in \Xh_2$ satisfies $\m_t\rng{i}{i+1}=\1\1$ then,
$\Ph_{\muhat}$-a.s., there must be a time $t_0\le t$ with
$\m_{t_0}\rng{i}{i+1}=\1\1$ but $\m_{t_0-1}\rng{i}{i+1}\ne\1\1$, and it
is easy to see by considering various possible values of
$\n_{t_0-1}\rng{i-1}{i+1}$ that this cannot happen. \end{proof}

In the remainder of \csect{rhobig} we study the ETIS states of the SSM
considered as a process with state space $\Xha$.  \cthm{tricky} justifies
this choice for states with density $\rhohat>1$.  Moreover, 
the state $\muhat^{(2)}$ introduced in \csect{intro} was shown in
\csect{rhohalf} to be an ETIS state of density $\rhohat=1$ for the SSM;
$\muhat^{(2)}$ is supported on $\Xha$ (and is the only such state).  It
is convenient to include this state in our study, so that from now on we
assume that $\rhohat\ge1$.

The dynamical rules of the SSM simplify on $\Xha$ as follows: if
$\n_t(i)$ is $\0$ or $\1$ then in the transition to $\n_{t+1}$ a particle
moves from each of the sites ${i+1}$ and ${i-1}$ to $i$, while if
$\n(j),\n(j+1)\ge\2$ a particle moves from $j$ to $j+1$ or the reverse,
each with probability $1/2$.  Since at each time step a particle must
move across each bond, the parity $(-1)^{\n(i)}$ of the stack height at
each site $i$ is conserved.  Let $S:=\{0,1\}^\bbz$ be the space of {\it
parity sequences}, define the {\it parity map} $\P:\Xha\to S$ by
$\P(\n)(i)=(1-(-1)^{\n(i)})/2$, and for $\sigma\in S$ define the {\it
parity sector} $\Xha_\sigma\subset \Xha$ by
$\Xha_\sigma:=\P^{-1}(\{\sigma\})$, so that $\n\in \Xha_\sigma$ iff, for
all $i$, $\n(i)$ and $\sigma(i)$ have the same parity.

Since for each $\sigma\in S$ the parity sector $\Xha_\sigma$ is invariant
under the dynamics, we obtain by restriction a dynamical system on each
$\Xha_\sigma$.  In \csect{even} below we discuss the stationary states in
the {\it even sector} $\Xha_e$, where $e$ is the parity sequence
satisfying $e(i)=0$ for all $i$, and in \csect{stkmeas} and \capp{others}
show how these give rise to all the ETIS states on $\Xha$.

\subsection{ETIS states in the even sector\label{even}}

Our next result describes a family of ETIS states for the SSM in the
even sector; the proof will be given shortly.

\begin{theorem}\label{unique} For each $\rhohat_e\ge1$ there is an
ETIS state $\mu^{(\rhohat_e)}_e$ on $\Xha_e$, of density
$\rhohat_e$; $\mu^{(\rhohat_e)}_e$ is a grand-canonical Gibbs state for
the statistical-mechanical system with state space $\Xha$ and one-body
potential
 \be\label{defV}
V(n):=2\ln2\,\delta_{n0}.
 \ee
  For $\rhohat_e=1$, $\mu_e^{(\rhohat_e)}=\muhat^{(2)}$ and hence is
ergodic but not weakly mixing.  For $\rhohat_e>1$, $\mu^{(\rhohat_e)}_e$
is mixing.  \end{theorem}

As we will see in \csect{stkmeas} and \capp{others}, for a full discussion of
the ETIS states of the SSM on $\Xha$ we need to know all stationary
states on $\Xha_e$, both TI and non-TI (should any of the latter exist).
We make the

\begin{conjecture}\label{main}For each $\rhohat_e\ge1$,
$\mu_e^{(\rhohat_e)}$ is the unique stationary state of the SSM on
$\Xha_e$ with density $\rhohat_e$.  \end{conjecture}

\noindent
In \csect{stkmeas} we will discuss the ETIS states on $\Xha$ under the
assumption that \cconj{main} holds, and in \capp{others} turn to the
general case.

We begin our discussion of \cthm{unique} with the consideration of the
SSM on a ring of $2L+1$ sites, indexed by $I_L:=\{-L,\ldots,L\}$; the
configuration space $\XhaL_e$ is the set of elements $\n\in\bbz_+^{I_L}$
for which $\n(i)$ is even for all $i$ and for which no two adjacent sites
both have height zero.  For the moment we take a fixed number $N$ of
particles, with $N$ even and $N\ge2L+2$, with corresponding configuration
space $\XhaNL_e\subset\XhaL_e$.  For $\n\in\XhaNL_e$ let $z(\n)$ be the
number of sites $i$ with $\n(i)=\0$.  In a transition from $\n_t$ to
$\n_{t+1}$ the direction of particle movement across $2z(\n_t)$ bonds is
determined and across the remainder is chosen randomly; moreover, a given
transition can occur via at most one set of these choices unless
$\n_t=\n_{t+1}$, in which case there are two possibilities (this occurs
iff $z(\n_t)=0$).  Thus the probability $P(\n_t,\n_{t+1})$ of such a
transition, if nonzero, is $2^{-(2L+1-2z(\n_t))+\delta_{\n_t,\n_{t+1}}}$.
Then a TIS state $\muLN$ is given by $\muLN(\n)=\ZLN^{-1}2^{-2z(\n)}$,
with $\ZLN^{-1}$ a normalizing constant, since because $P(\n,\n')=0$ iff
$P(\n',\n)=0$, $\muLN$ satisfies the detailed balance condition
$\muLN(\n)P(\n',\n)=\muLN(\n')P(\n,\n')$.  It is straightforward to check
that the dynamics permits transition from any configuration in $\XhaNL_e$
to any other, so that $\muLN$ is the unique TIS state.

The state $\muLN$ is a Gibbs measure arising from the one-particle
potential $V(n)$ of \eqref{defV}, that is,
$\muLN(\n)=\ZLN^{-1}\prod_{i=-L}^Le^{-V(\n(i))}$. (One may also view
$\muLN$ as a Gibbs measure on the space of all $N$-particle
configurations, with one- and two-body hard core potentials, that is,
formal potentials taking infinite values, that impose the restrictions of
$\XhaL_e$.)  In order to pass to the $L\to\infty$ limit it is convenient
to consider the grand canonical measure with fugacity $\zeta\ge0$:
 \be\label{muzetaL}
\muLz:=\Xi_{L,\zeta}^{-1}\sum_{\n\in\XhaL_e}\zeta^{\sum_i\n(i)-2L-2}
e^{-\sum_iV(\n(i))},
 \ee
 with $\Xi_{\zeta,L}^{-1}$ again a normalizing constant.

\begin{lemma}\label{Gibbs} The limiting measure
$\muz=\lim_{L\to\infty}\muLz$ exists for $0\le\zeta<1$ and is a TIS
state of density $1/(1-\zeta)$ for the SSM on $\Xha_e$.  Moreover,
$\mu^{(0,\infty)}=\mu^{(2)}$ and $\mu^{(\zeta,\infty)}$ is mixing if
$\zeta>1$.  \end{lemma}

\begin{proof} The case $\zeta=0$ follows immediately from
the fact, evident from \eqref{muzetaL}, that $\mu^{(0,L)}$ gives equal
weight to the $2L+1$ configurations in $\XhaNL_e$ with
$\sum_i\n(i)=2L+2$.  From now on we assume that $0<\zeta<1$.

 We can prove the existence of $\muz$, and also calculate many of its
properties, using the standard transfer matrix formalism; we give only a
sketch.  Let us think of
$\ell^2=\{(x_i)_{i=0}^\infty\mid\sum x_i^2<\infty\}$ as a space of column
vectors, with $u^T$ denoting the transpose of the vector $u$, and define
$u,v\in\ell^2$ by $u_i:=\delta_{i0}$ and $v_0:=0$, $v_i:=\zeta^i$ if
$i\ge1$.  Let $\T:=(uv^T+vu^T)/2+vv^T$; then for $\n=2\i\in \XhaL_e$,
 \be
\muLz(\n) = \Xi_{\zeta,L}^{-1}\T_{\i(-L)\i(-L+1)}\T_{\i(-L+1)\i(-L+2)}
  \cdots \T_{\i(L-1)\i(L)}\T_{\i(L)\i(-L)},
 \ee
 with $\Xi_{\zeta,L}$ the trace of $\T^{2L+1}$.  $\T$ is a rank 2 operator
with nonzero eigenvalues $\lambda_1=\zeta/(2(1-\zeta))$ and
$\lambda_2=-\zeta/(2(1+\zeta))$; the eigenvector associated with
$\lambda_1$, the larger in magnitude, is $w=(\zeta/(1+\zeta))u+v$, so
that
 \be\label{limit}
 \lim_{n\to\infty}\lambda_1^{-n}\T^n=\|w\|^{-2}w^Tw.
 \ee
   Thus for $\m=2\i\in\bbz_+^{\{-K,\cdots K\}}$ and $E$ the event that
$\n\rng{-K}K=\m\rng{-K}K$ we have
 \begin{align}
 \muz(E)\nonumber
  &:=\lim_{L\to\infty}\muLz(E)\\
  &={\cal Y}_{\zeta,K}^{-1}
    w_{\i(-K)}\T_{\i(-K)\i(-K+1)}\cdots \T_{\i(K-1)\i(K)}w_{\i(K)}\label{muz}
 \end{align}
 with 
 \be\label{calY}
{\cal Y}_{\zeta,K}=\lambda_1^{2K}\|w\|^2
  = \frac{\zeta^{2K+2}}{(1+\zeta)^22^{2K-1}(1-\zeta)^{2K+1}}.
 \ee
 Taking $K=0$ in \eqref{muz} we find that 
 \be\label{ssd}
\muz(\n(0)=2\i)={\cal Y}_{\zeta,0}^{-1}w_{\i}^2
 = \begin{cases}(1-\zeta)/2,&\text{if $\i=0$},\\
  (1+\zeta)^2(1-\zeta)\zeta^{2\i-2}/2,&\text{otherwise,}
 \end{cases}
 \ee
 from which we find the density
$\rhohat(\zeta):=\muz(\n(0))=1/(1-\zeta)$.  From \eqref{limit} and
\eqref{muz} it follows that if $f,g:\Xha\to\bbr$ each depend only on the
values of the configuration at a finite number of sites then
$\muz(f\tau^ng)\sim Ce^{-(\lambda_1-|\lambda_2|)n}$ as $n\to\infty$, so
that $\muz$ is mixing.

 The stationarity of $\muz$ can be verified from the explicit formulas
\eqref{muz} and \eqref{calY}, but it is simpler to argue from the
stationarity of $\muLz$.  Take $L>K>0$, suppose that
$A^{(K)}\subset \XhaK_e$, and let $A$, respectively $A^{(L)}$, be the
set of $\n\in \Xha_e$, respectively $\n\in\XhaL_e$, such that
$\n\rng{-K-1}{K+1}\in A^{(K)}$.  If $Q$ and $Q^{(L)}$ are the
transition kernels for the SSM on $\Xha$ and $\XhaL$, respectively, then the
stationarity of $\muLz$ implies that
 \be\label{Lsta}
\int_{\XhaL}Q^{(L)}(\n,A^{(L)})\muLz(d\n)=\muLz(A^{(L)}.
 \ee
 Now $Q^{(L)}(\n,A^{(L)})$ depends only on $\n\rng{-K-1}{K+1}$ and
$Q^{(L)}(\n,A^{(L)})=Q(\n',A)$ if $\n\rng{-K-1}{K+1}=\n'\rng{-K-1}{K+1}$,
so that taking the $L\to\infty$ limit in \eqref{Lsta} yields
$\int_{\Xha}Q(\n,A)\muz(d\n)=\muz(A)$, the stationarity of $\muz$.
\end{proof}

\begin{proofof}{\cthm{unique}}The theorem follows immediately from
\clem{Gibbs} via
$\nu^{(\rhohat_e)}_e=\nu^{((\rhohat_e-1)/\rhohat_e),\infty)}$.
\end{proofof}

\begin{remark}\label{alt} There is an alternative way to describe the
state $\muz$ (and hence also $\mu_e^{\rho_e}$).  Consider the image of
$\muz$ under the map $F:\Xha\to\{0,1\}^\bbz$ defined by
$F(\n)(j)=\min\{1,\n(j)\}$, which effectively classes sites simply as
occupied or empty.  The image measure $F_\muz$ is again Gibbisan, with
no interactions other than the exclusion of configurations containing
adjacent holes, so that it is, after an interchange of the roles of
particles and holes, the equilibrium state of the familiar {\it
nearest-neighbor hard core} model.  The holes in this system have
effective fugacity $(1-\zeta^2)/(2\zeta)^2$ (relative to a fugacity of 1
for the particles) and from \eqref{ssd} the density of holes is
$(1-\zeta)/2$.  The full state $\muz$ is then obtained by first
conditioning $F(\n)=\eta$ for some $\eta\in\{0,1\}^\infty$ with no
adjacent holes, distributed according to $F_*\muz$, and then distributing
particles on each site $j$ for which $\eta(j)=1$ independently, with
distribution $\mu\{\n(j)=2i\}=\zeta^{2i-1}/(1-\zeta^2)$,
$i=1,2,\ldots$.\end{remark}

\subsection{ETIS states for  the SSM\label{stkmeas}}

We now discuss the passage from stationary states of the SSM on the even
sector to general high-density ETIS states, specifically, to states with
$\rhohat>1$ as well as the special state $\mu^{(2)}$ with $\rhohat=1$.
By \cthm{tricky}, these are precisely the ETIS states on $\Xha$.  Let us
define $\gamma:\Xha_e\times S\to \Xha$ by $\gamma(\m,\sigma)=\m+\sigma$
(with component-wise addition: $(\m+\sigma)(i):=\m(i)+\sigma(i)$).  Note
that for fixed $\sigma\in S$, $\gamma(\cdot,\sigma)$ is a bijection of
$\Xha_e$ with $\Xha_\sigma$.  We define the dynamics in $\Xha_e\times S$
to be constant on $S$.

\begin{lemma}\label{gamma} (a) A measure $\mu$ on $\Xha$ is an ETIS state
for the SSM iff $\tilde\mu:=\gamma^{-1}_*\mu$ is an ETIS state on
$\Xha_e\times S$. 
 \par\noindent
(b) Each ETIS state $\tilde\mu$ of density $\rhohat$ on
$\Xha_e\times S$ has the form
$\tilde\mu(d\n\,d\sigma)=\mu_\sigma(d\n)\lambda(d\sigma)$, where
$\lambda$ is an ergodic TI probability measure on $S$ and the
$\mu_\sigma$, $\sigma\in S$, are stationary probability measures on
$\Xha_e$ satisfying $\mu_{\tau\sigma}=\tau_*\mu_\sigma$.  Moreover, if
$\lambda$ has density $\kappa$ then $\lambda$-almost all $\mu_\sigma$
have density $\rhohat_e=\rhohat-\kappa$.\end{lemma}

\begin{proof} (a) The map $\gamma$ clearly commutes with translations and
with the dynamics.  Thus $\gamma_*$ and $\gamma^{-1}_*$ carry TIS states
to TIS states, and they clearly preserve extremality.

 \par\noindent
 (b) The form $\tilde\mu(d\n\,d\sigma)=\mu_\sigma(d\n)\lambda(d\sigma)$ is
immediate, with $\lambda$ the probability measure on $S$ giving the
distribution of $\sigma$ and $\mu_\sigma$ the conditional probability
measure on $\Xha_e$ given $\sigma$.  Stationarity of $\tilde\mu$ implies
stationarity of each $\mu_\sigma$, and translation invariance of
$\tilde\mu$ yields translation invariance of $\lambda$ and the relation
$\mu_{\tau\sigma}=\tau_*\mu_\sigma$.  $\lambda$ must be ergodic, since a
decomposition of $\lambda$ as a convex combination of TI measures would
yield immediately a decomposition of $\mu$ in terms of TIS measures.
Finally, since $\tilde\mu$-a.s.~configurations have density $\rhohat$,
$\lambda$-a.s.~each $\mu_\sigma$ must have density
$\rhohat_e=\rhohat-\kappa$.  \end{proof}

Throughout the remainder of this section we assume that \cconj{main}
holds.  See \capp{others} for an analysis of the situation when this
assumption is not valid.

\begin{theorem}\label{structure} Suppose that the $\mu_e^{(\rhohat_e)}$
of \cthm{unique} are the only stationary states of the SSM on $\Xha_e$,
i.e., that \cconj{main} holds.  Then the ETIS states with density
$\rhohat>1$ are precisely the states
$\mu^{(\rhohat,\lambda)}:=\gamma_*(\mu^{(\rhohat_e)}_e\times\lambda)$,
with $\lambda$ an ergodic TI probability measure on $S$ of density
$\kappa\le\rhohat-1$ and $\rhohat_e=\rhohat-\kappa$.
\end{theorem}

Before giving the proof we note an immediate consequence of this result
and \cthm{bijection}.

\begin{corollary}\label{FSSEPhigh}Under the hypotheses of
\cthm{structure} the ETIS states of the F-SSEP with density $\rho>1/2$
are the states $\Phi_\phi\bigl(\mu^{(\rhohat,\lambda)}\bigr)$;
$\Phi_\phi\bigl(\mu^{(\rhohat,\lambda)}\bigr)$ has density
$\rhohat/(1+\rhohat)$. \end{corollary}

\begin{proofof}{\cthm{structure}}By \clem{gamma} it suffices to prove
that the ETIS states on $\Xha_e\times S$ are the states
$\mu^{(\rhohat_e)}_e\times\lambda$ described in the theorem.  The latter
are clearly TIS.  If $\tilde\mu$ is an ETIS state on $\Xha_e\times S$ of
density $\rhohat$ then \clem{gamma} and \cconj{main} imply that
$\tilde\mu=\mu^{(\rhohat-\kappa)}_e\times\lambda$.  Conversely, a
decomposition of a state $\mu^{(\rhohat_e)}_e\times\lambda$ as described
in the theorem into extremal components must, by this and the ergodicity
of $\lambda$, be trivial.\end{proofof}

\begin{remark}\label{ergodicity} (a) The states $\mu^{(\rhohat,\lambda)}$
of the SSM with $\rhohat>1$ are defined in \cthm{structure}; it is
natural then to define also
$\mu^{(1,\delta_e)}:=\gamma_*(\mu^{(\rhohat_e)}_e\times\delta_e)=\mu^{(2)}$,
with $\delta_e\in\Mb(S)$ the point mass on the zero configuration $e$.

 \smallskip\noindent
 (b) It follows from the theorem that, whenever $\mu^{(\rhohat,\lambda)}$
is defined, the density $\kappa$ of the measure $\lambda$ satisfies
$0\le\kappa\le\max\{1-\rhohat,1\}$.  The case $\rhohat=1+\kappa$,
corresponding, in the notation of the theorem, to $\rhohat_e=1$, is of
particular interest; see (c) and (d) below.

 \smallskip\noindent
 (c) $\mu^{(\rhohat,\lambda)}$ is extremal in the class of TIS states,
but it may not be ergodic under translations.  For example, if
$\sigma\in S$ has even period $p$ and density $\kappa$, and $\lambda$ is
the superposition of the point masses on $\sigma$ and its translates,
then if $\rhohat=1+\kappa$ (i.e., if $\rhohat_e=1$),
$\mu^{(\rhohat,\lambda)}$ is a superposition of two ergodic measures,
these being the superpositions of the translates of $\n^*+\sigma$ and
$\tau\n^*+\sigma$, respectively.  When in general can such non-ergodicity
occur?  First, if $\rhohat_e>1$ then $\mu^{(\rhohat_e)}_e\times\lambda$,
as the product of a mixing and an ergodic measure, is ergodic, and hence
so is $\mu^{(\rhohat_e+\kappa,\lambda)}$.  On the other hand, if
$\rhohat_e=1$ we use the fact that $\mu^{(\rhohat_e)}_e\times\lambda$
will be ergodic iff the eigenvalue 1 of the translation operator in
$L^2(\Xha_e\times S,\mu^{(\rhohat_e)}_e\times\lambda)$ is simple; since
translation on $L^2(\Xha_e,\mu^{(1)}_e)$ has simple eigenvalues of $\pm1$
and no others, we conclude that $\mu^{(\rhohat,\lambda)}$ is ergodic iff
the translation operator on $L^2(S,\lambda)$ does not have the eigenvalue
$-1$.

 \smallskip\noindent
 (d) If $\lambda$ has density $\kappa$ then,
$\mu^{(1+\kappa,\lambda)}$-a.s., each configuration $\n$ has
sites with 0 or 1 particle(s) alternating with sites with 2 or 3
particles, and the dynamics carries such a configuration to the one
obtained from it by the substitutions $0\to2$, $1\to3$, $2\to0$, and
$3\to1$. The evolution of these configurations is thus periodic, of
period 2.
\end{remark}

We finally observe that the states $\mu^{(\rhohat,\lambda)}$ for fixed
$\rhohat_e=\rhohat-\kappa$ (where as usual $\kappa$ is the density of
$\lambda)$ form a $\lambda$-family (see \cdef{lamfam}), which we denote
$\Fch_{\rhohat_e}$.  The indexing map
$\Psih_{\rhohat_e}:\L\to\Fch_{\rho_e}$ is given by
$\Psih_{\rhohat_e}(\lambda)=\mu^{(\rhohat_e+\kappa,\lambda)}$, where we
have made the identification $\L=\Mb(S)$.  From \ccor{FSSEPhigh}, then,
the ETIS states of the F-SSEP with density $\rho>1/2$, together with the
state $\mu^{(2)}$, form the $\lambda$-families
$\Fc_{\rhohat_e}=\Phi_*(\Fch_{\rhohat_e})$.
 
\section{Summary\label{conclude}}

The results of this paper are summarized, under the assumption that
\cconj{main} holds, by \cfig{symbolic}, which gives a symbolic depiction
of the relations among the $\lambda$-families of regular ETIS states of
the SSM. The heavy black lines denote $\lambda$-families.
$\Fch_{\rm left}$ and $\Fch_{\rm right}$ have been separated for
visibility, but in fact both lie at $\rhohat=1$.  The state
$\muhat^{(1)}$ belongs to all three of the families $\Fch_{\rm low}$,
$\Fch_{\rm left}$, and $\Fch_{\rm right}$, and the state $\muhat^{(2)}$
to $\Fch_{\rm left}$, $\Fch_{\rm right}$, and $\Fch^{(1)}$ (see
\crem{ergodicity}(a)).  $\kappa$ denotes the density of a measure
$\lambda\in\Mb(S)$; this variable is relevant only for the state
$\muhat^{(2)}$ (for which $\kappa=0$) and for states with $\rhohat>1$.
The shaded region is filled with the $\lambda$-families
$\Fch_{\rhohat_e}$; the family for $\rhohat_e=1$ and three other
representative families ($\rhohat_e=r_1,r_2,r_3$) are shown.

\begin{figure}[ht]
  \centering
\includegraphics[scale=0.92]{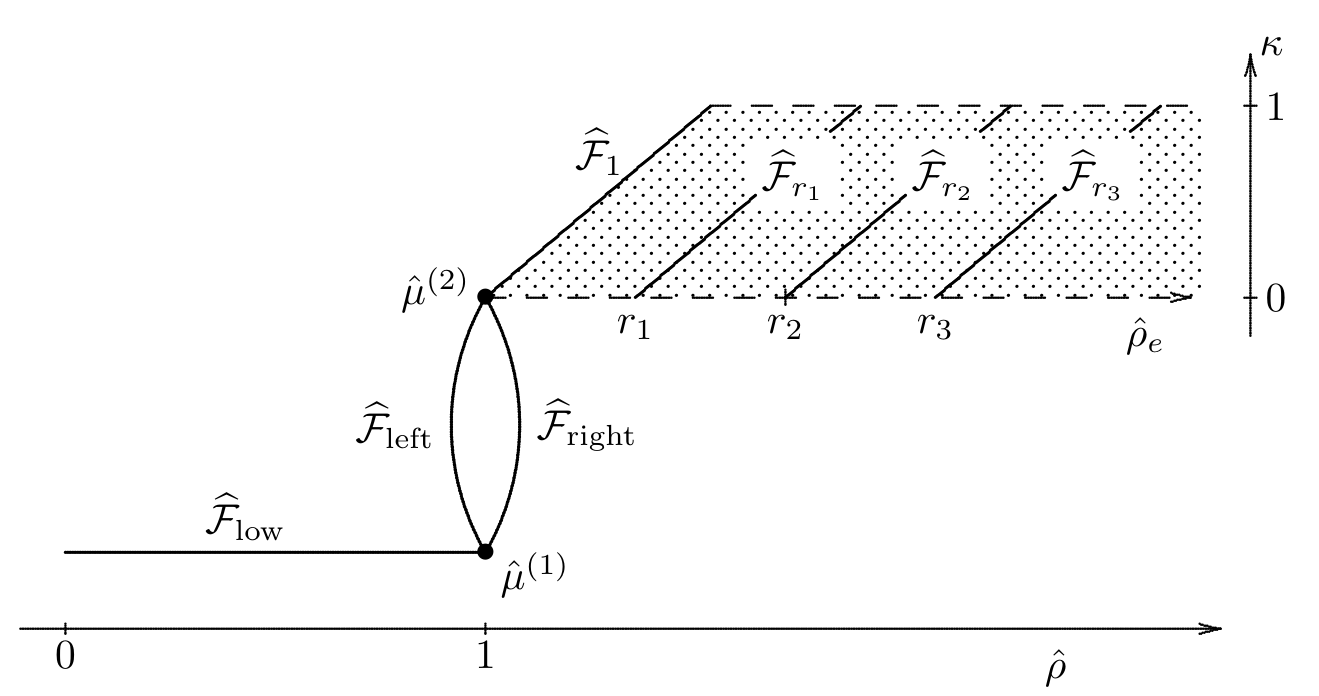}
 \caption{Symbolic picture of the set of ETIS states of the SSM.}
 \label{symbolic} \end{figure}

 \medskip\noindent
  {\bf Acknowledgments:} The work of JLL was supported by the AFOSR under
award number FA9500-16-1-0037 and Chief Scientist Laboratory Research
Initiative Request \#99DE01COR.  

\appendix

\section{Equivalence of TI measures under substitutions\label{subs}}

In this appendix we give a construction which will be used at several
points in the paper.  Suppose that $\S$ and $\T$ are countable alphabets,
that $\Yh=\S^\bbz$ and $Y=\T^\bbz$ with typical elements $\zeta\in\Yh$
and $\eta\in Y$, and that for each $s\in \S$ we specify a finite
sequence $\chi(s)=t_s(1)\ldots t_s(k(s))$ of elements of $\T$ (that is,
a word of $k(s)$ letters in the alphabet $\T$).  Then we define
$\phi:\Yh\to Y$ to be the map which substitutes $\chi(s)$ for $s$, that
is,
 \be
\phi(\ldots \zeta(-1)\zeta(0)\zeta(1)\zeta(2)\ldots)
    =\ldots\chi(\zeta(-1))\chi(\zeta(0))\chi(\zeta(1))\chi(\zeta(2))\ldots,
   \label{fund}
 \ee
 with $\chi(\zeta(1))$ beginning at site 1, so that
$\phi(\zeta)(1)=t_{\zeta(1)}(1)$.  For $s\in \S$ we define
$\Yh_s=\{\zeta\in \Yh\mid \zeta(1)=s\}$ and $Y_s=\phi(\Yh_s)$; for
$0\le j\le k(s)-1$ we set $Y_{sj}=\tau^{-j}Y_s$ and define
$\phi_{sj}:\Yh_s\to Y_{sj}$ by $\phi_{sj}=\tau^{-j}\phi\big|_{\Yh_s}$.
From now on we assume that the $\chi(s)$ are such that the sets $Y_{sj}$
are pairwise disjoint.

  We call a TI measure $\nuhat$ on $\Yh$ {\it regular} if
$\nuhat(k(\zeta(1)))$ is finite, i.e., if
$ Z_{\nuhat}:=\sum_{s\in \S}k(s)\nuhat(\Yh_s)<\infty$, and a TI measure
$\nu$ on $Y$ {\it regular} if it is supported on
$\bigcup_{s\in \S}\bigcup_{j=0}^{k(s)-1}Y_{sj}$, the smallest translation
invariant subset of $\T^\bbz$ containing $\phi(\Yh)$.  Let $\M(\Yh)$
denote the space of regular TI states on $\Yh$ and $\M(Y)$ the space of
regular TI states on $Y$.  If $\nuhat\in\M(\Yh)$ and $\nu\in\M(Y)$ define
$\nuhat_s=\nuhat\big|_{\Yh_s}$ and $\nu_s=\nu\big|_{Y_s}$.

\begin{theorem}\label{Phithm} {\sl For $\nuhat\in\M(\Yh)$
 define the measure $\Phi(\nuhat)$ on $Y$ by
 \be
\Phi(\nuhat):=Z_{\nuhat}^{-1}
   \sum_{s\in \S}\sum_{j=0}^{k(s)-1}\phi_{sj*}\nuhat_s.\label{Phi}
 \ee
 Then $\Phi$ is a bijection of $\M(\Yh)$ with $\M(Y)$, and for
$\nu\in\M(Y)$,
$\Phi^{-1}(\nu)=\widetilde
Z_\nu^{-1}\sum_{s\in \S}\phi_{s0\,*}^{-1}\nu_s$, where $\widetilde
Z_\nu=\sum_{s\in \S}\nu(Y_s)$.  $\Phi(\nuhat)$ is ergodic iff $\nuhat$
is.} \end{theorem}
 
\begin{proof}This is straightforward to check.  The final statement
follows from the fact that $\nuhat=\sum c_\alpha\nuhat_\alpha$ if and
only if $\Phi(\nuhat)=\sum c_\alpha'\Phi(\nuhat_\alpha)$, where
$\nuhat_\alpha\in\M(\Yh)$ and
$c_\alpha'=Z_{\nuhat_\alpha}c_\alpha/Z_{\nuhat}$.\end{proof}

We now suppose that we are given dynamical rules in the spaces $\Yh$ and
$Y$, that is, translation invariant Markov processes with state spaces
$\Yh$ and $Y$, specified by respective TI transition kernels
$\Qh(\zeta,A)$ and $Q(\eta,B)$.  We will say that $\Qh$ or $Q$ {\it
preserves ergodicity} if $\nuhat\Qh$ (respectively $\nu Q$) is ergodic
whenever $\nuhat$ (respectively $\nu$) is; {\it preserving regularity} is
defined similarly.

We next want to give a condition which will imply that the mapping $\Phi$
preserves these dynamics, i.e., that $\Phi(\nuhat\Qh)=\Phi(\nuhat)Q$; we
will need some further notation.  Let $\Yo\subset \Yh\times Y$ be the set
of pairs $(\zeta,\eta)$ such that $\eta$ is a (possibly trivial)
translate of $\phi(\zeta)$, and let $\pihat$ and $\pi$ be the projections of
$\Yo$ onto the first and second components, respectively, of
$\Yh\times Y$.

\begin{definition}\label{taucouple} A {\it $\tau$-coupling} of $\Qh$ and
$Q$ is a  Markov transition kernel $\Qo$ with state
space $\Yo$ such that for $(\zeta,\eta)\in\Yo$, $\Ahat\subset\Yh$, and
$A\subset Y$,
 \begin{align}\label{toAhat}
\Qo\bigl((\zeta,\eta),\pihat^{-1}(\Ahat)\bigr)=\Qh(\zeta,\Ahat),\\
\Qo\bigl((\zeta,\eta),\pi^{-1}(A)\bigr)=Q(\eta,A).\label{toA}
 \end{align}
  Equivalently, a Markov transition kernel $\Qo$ with state space $\Yo$
is a $\tau$-coupling of $\Qh$ and $Q$ provided that for any Markov
process $(\zeta_t,\eta_t)$ with transition kernel $\Qo$, $\zeta_t$
and $\eta_t$ are Markov processes with transition kernels $\Qh$ and $Q$,
respectively.  \end{definition}

\begin{remark}\label{shel}To show that there is a $\tau$-coupling of $Q$
and $\Qh$ it suffices to find a restricted transition probability
$\Qo\bigl((\zeta,\phi(\zeta)), \cdot\bigr)$ which satisfies
\eqref{toAhat} and \eqref{toA} when $\eta=\phi(\zeta)$.  For $\Qo$ can
then be extended to the rest of $\Yo$ by setting
$\Qo\bigl((\zeta,\eta),\cdot\bigl)
:=\tauo^q_*\Qo\bigl((\zeta,\phi(\zeta)),\cdot\bigl)$
when $\eta=\tau^q\phi(\zeta)$, choosing $q$, when periodicity of
$\phi(\zeta)$ necessitates a choice, to be the minimal nonnegative $q$
with $\eta =\tau^q\phi(\zeta)$ .  Here $\tauo$ acts on $\Yo$ via
$\tauo(\zeta,\eta):=(\zeta,\tau\eta)$.\end{remark}

\begin{theorem}\label{equivariant} Suppose that $Q$ and $\Qh$ are TI
Markov transition kernels on $\Yh$ and $Y$, respectively, which preserve
ergodicity and regularity and for which there exists a $\tau$-coupling
$\Qo$.  Then for any TI state $\nuhat\in\M(\Yh)$ and any $n\ge1$,
$\Phi(\nuhat\Qh^n)=\Phi(\nuhat)Q^n$.  \end{theorem}

\begin{corollary}\label{stationary} If $Q$ and $\Qh$ are as in
\cthm{equivariant} then $\Phi$ is a bijection of $\M_s(\Yh)$ with
$\M_s(Y)$, i.e., $\nuhat\in\M(\Yh)$ is stationary for $\Qh$ if and only
if $\Phi(\nuhat)$ is stationary for $Q$.  \end{corollary}

The corollary is of course immediate.  The idea of the proof of
\cthm{equivariant} is taken from \cite{AGLS} and \cite{GLS2}:

\begin{proofof}{\cthm{equivariant}}It suffices to verify the result for
$\nuhat$ ergodic.  Then since $\Qh$ and $Q$ preserve ergodicity, as does
$\Phi$, $\Phi(\nuhat\Qh^n)$ and $\Phi(\nuhat)Q^n$ are ergodic, so that
these two measures are either equal or mutually singular.  Hence to prove
their equality it suffices to find a nonzero measure $\lambda$ with
$\lambda\le\Phi(\nuhat\Qh^n)$ and $\lambda\le\Phi(\nuhat)Q^n$.

Let $\Qo$ be a $\tau$-coupling of $\Qh$ and $Q$ with state space $\Yo$,
as in \cdef{taucouple}, define $\psi:\Yh\to\Yo$ by
$\psi(\zeta):=(\zeta,\phi(\zeta))$, and for $\nuhat\in\M(\Yh)$ let
$\nuo:=\psi_*\nuhat$, so that $\pihat_*\nuo=\nuhat$ and
$\pi\nuo=\phi_*\nuhat$.  Fix $n\ge1$ and let $q\in\bbz$ be such that
$(\nuo\Qo^n)(C_q)>0$, where
$C_q:=\{(\zeta,\eta)\in\Yo\mid\eta=\tau^q\phi(\zeta)\}$, and define
$\nuo^{(n,q)}:=\ind_{C_q}\nuo\Qo^n$, $\nu':=\phi_*\pihat_*\nuo^{(n,q)}$,
and $\nu'':=\pi_*\nuo^{(n,q)}$.  We claim that (i)~$\nu''=\tau^q_*\nu'$
and that, for an appropriate constant $c>0$,
(ii)~$c\nu'\le\Phi(\nuhat\Qh^n)$ and (iii)~$c\nu''\le\Phi(\nuhat)Q^n$.
Since $\Phi(\nuhat)Q^n$ is TI, the proof is completed by taking
$\lambda=c\nu'$.

It remains to prove the claim.  (i) follows from the definition of $C_q$.
 From \cdef{taucouple}, $\pihat_*(\nuo\Qo^n)=\nuhat\Qh^n$ and
$\pi_*(\nuo\Qo^n)=(\phi_*\nuhat)Q^n$, and with this and \eqref{Phi}
we have
 \be\label{nup}
\nu'=\phi_*\pihat_*\nuo^{(n,q)}\le\phi_*\pihat_*(\nuo\Qo^n)
  =\phi_*(\nuhat\Qh^n)\le Z_{\nuhat\Qh^n}\Phi(\nuhat\Qh^n)
 \ee
and
 \be\label{nudp}
 \nu''=\pi_*\nuo^{(n,q)}\le\pi_*(\nuo\Qo^n)
  =(\phi_*\nuhat) Q^n\le Z_{\nuhat}\Phi(\nuhat)Q^n.
 \ee
 This verifies parts (ii) and (iii) of the claim, with
$c:=\min\{Z_{\nuhat\Qh^n}^{-1},Z_{\nuhat}^{-1}\}$.\end{proofof}

In the remainder of this appendix we discuss the applications we make of
these results.  \cthm{Phithm} is used in \csect{rhohalf}; the
substitution maps there are denoted $\phileft$ and $\phiright$.
\cthms{Phithm}{equivariant} are used to obtain a correspondence between
the stationary states of the Symmetric Stack Model (SSM) and of the
Facilitated Simple Symmetric Exclusion Process (F-SSEP), a correspondence
first introduced in \csect{intro} (see that section and \csect{prelims}
for the definition of the transition kernels $\Qh$ and $Q$ for these
models) and used throughout the paper.  In this application $\S=\bbz_+$,
$\T=\{0,1\}$, and for $n\in\bbz_+$, $\chi(n):=0\,1^n$; we write
$\Xh:=\bbz_+^\bbz$ and $X=\{0,1\}^\bbz$ but keep the notation
$\phi:\Xh\to X$ for the substitution map obtained from $\chi$.  The
general definition of regularity of measures given above corresponds in
this case to the definition given in \csect{prelims}, and it is clear
that $\Qh$ and $Q$ preserve regularity.

\begin{lemma}\label{QQhatergodic}$\Qh$ and $Q$ preserve ergodicity.
\end{lemma}

\begin{proof}We consider $\Qh$ (the proof for $Q$ is similar), and so
must show that if $\nuhat\in\Mb(\Xh)$ then $\nuhat\Qh$ is ergodic.  We
define the probability space $(\Omega,P)$ by
$\Omega:=\Xh\times\{0,1\}^\bbz$, $P:=\nuhat\times\kappa$, where $\kappa$
is the Bernoulli measure with parameter $1/2$, and write a typical
element of $\Omega$ as $(\n,\alpha)$.  Then we can introduce a concrete
realization on $\Omega$ of one step of the $\Qh$ process, from $\n_0$
distributed as $\nuhat$ to $\n_1$ distributed as $\nuhat\Qh$, as follows.
Recall that if $\n_0$ has a short stack on either $i$ or $i+1$ then the
movement, or non-movement. of a particle across the bond
$\langle i,i+1\rangle$, in passing from $\n_0$ to $\n_1$, is determined
by the rule given in \csect{intro}; we supplement this rule by requiring
that if $\n_0$ has tall stacks at both $i$ and $i+1$ then a particle moves
from site $i$ to site $i+1$ if $\alpha(i)=1$ and from $i+1$ to $i$ if
$\alpha(i)=0$.

As the product of measures which are respectively ergodic and mixing
under translations, $P$ is ergodic under translations.  But then
$\nuhat\Qh=\n_{1*}P$ is the covariant image of an ergodic measure and
hence is ergodic.\end{proof}

\begin{theorem}\label{QQhat}There exists a $\tau$-coupling $\Qo$ of the
  Markov transitions kernels $\Qh$ and $Q$ for the SSM and F-SSEP.
\end{theorem}

\begin{proof}As in the definition of $\Qh$ and $Q$ in \csect{prelims} we
give the transition rules for a Markov process $(\n_t,\eta_t)$ on
$\Xo\subset \Xh\times X$ (see \cdef{QQhat}), leaving the specification of
$\Qo$ as an easy exercise.  By \crem{shel} it suffices to consider only
the transition from $(\n_0,\eta_0)$ to $(\n_1,\eta_1)$ for
$\eta_0=\phi(\n_0)$.  As a preliminary, for $\n\in\Xh$ we define
the map $K=K_{\n}$, with $K:\bbz\to\bbz$, so that $K(i)$ is the
starting point of the word $\chi(\zeta(i))$ in the substitution
\eqref{fund}:
 \be
K(i)=\begin{cases}1+\sum_{j=1}^{i-1}(\n(j)+1),& \text{if $i\ge1$,}\\
    1-\sum_{j=i}^0(\n(j)+1),& \text{if $i\le0$.}\end{cases}
 \ee

For the dynamics under $\Qo$ we allow $\n_0$ to evolve to $\n_1$
according the the $\Qh$ dynamics, and let $\n_1$ and $\eta_0$ determine
$\eta_1$ as follows.  Write $K=K_{\n_1}$. Then, if in passing from $\n_0$
to $\n_1$ a particle jumps from site $i-1$ to site $i$, then in passing
from $\eta_0$ to $\eta_1$ a particle jumps from site $K(i)-1$ to site
$K(i)$, while if a particle jumps from site $i$ to site $i-1$ in passing
from $\n_0$ to $\n_1$ then one jumps from site $K(i)+1$ to site $K(i)$ in
passing from $\eta_0$ to $\eta_1$.  It is straightforward to verify that
then $(\n_1,\eta_1)\in\Xo$ and with $\eta_1$ distributed according to
$Q(\eta_0,\cdot)$.  \end{proof}

With \clem{QQhatergodic} and \cthm{QQhat} we can apply \ccor{stationary}
to obtain the correspondence of the F-SSEP and SSM stationary states.  The
result is summarized in \cthm{bijection}.

\section{Possible ETIS states of the SSM for $\rhohat>1$.\label{others}}

In this appendix we discuss, in the case where \cconj{main} is not
satisfied, the passage from stationary states of the SSM on the even
sector $\Xha_e$ to general ETIS states of the SSM with density
$\rhohat>1$.  By \cthm{tricky}, the latter all have support on $\Xha$;
moreover, as observed in \csect{rhobig}, the stationary states on $\Xha$
include in addition only $\muhat^{(2)}$.  Let $\N_{\rhohat_e}$ denote the
family of stationary states on $\Xha_e$ with density $\rhohat_e$, that
is, states $\nu$ for which $\nu$-a.e.~$\n\in\Xha_e$ satisfies
 \be\label{rhodefx}
\lim_{N\to\infty}\frac1N\sum_{i=1}^N\n_i =
\lim_{N\to\infty}\frac1N\sum_{i=-N}^{-1}\n_i = \rhohat_e.
 \ee
 Let $\Nb_{\rhohat_e}$ be the extremal elements of $\N_{\rhohat_e}$.  We
analyze completely the structure of ETIS states on $\Xha$, in terms
of the states of $\Nb_{\rhohat_e}$, in two cases: when for each
$\rhohat_e$, $\Nb_{\rhohat_e}$ contains only TI states
(\cthm{allti}), and when for each $\rhohat_e$, $\Nb_{\rhohat_e}$ is
countable (\cthm{structurex}).  The first case is quite simple:

\begin{theorem}\label{allti} If all stationary states on $\Xha_e$ are
translation invariant then every ETIS state $\mu$ on $\Xha$ of density
$\rhohat$ is of the form $\gamma_*(\mu_*\times\lambda)$, with $\lambda$
an ergodic TI measure on $S$ of density $\kappa$ and
$\mu_*\in\Nb_{\rhohat_e}$, where $\rhohat_e=\rhohat-\kappa$.  Conversely,
every state of this form is ETIS.  \end{theorem}

\begin{proof} If $\mu$ is an ETIS state on $\Xh^*$ of density $\rhohat$ then
\clem{gamma} implies that $\mu=\gamma_*\tilde\mu$ with
$\tilde\mu(d\n\,d\sigma)=\mu_\sigma(d\n)\lambda(d\sigma)$ and
$\mu_{\tau\sigma}=\tau_*\mu_\sigma$.  But if all stationary states on
$\Xha_e$ are TI then $\mu_\sigma=\mu_{\tau\sigma}$ for all $\sigma$, so
that the ergodicity of $\lambda$ implies that $\mu_\sigma$ is independent
of $\sigma$, $\lambda$-a.s., and extremality of $\mu$ implies that
this state, $\mu_*$, must be ETIS on $\Xha_e$.  The converse is clear.
\end{proof}

We now consider the possibility of non-TI stationary states on $\Xha_e$.
In the next definition we introduce a class of stationary states on
$\Xha_e\times S$ which we call {\it basic} states, and a further
restriction of this class to {\it irreducible} basic states.

\begin{definition}\label{answer} Let $\lambda$ be an ergodic TI measure
on $S$, let $\nu\in\Nb_{\rhohat_e}$ have period $n=n(\nu)$ under
translation, and let $m$ and $q$ be positive integers such that $n=qm$.
 \par\noindent
 (a) A {\it $(\lambda,m)$-partition} of $S$ is an ordered family
$A=\bigl(A_i\bigr)_{i=0}^{m-1}$ of subsets of $S$ such that
$\bigcup_iA_i=S$ and $A_i\cap A_j=\emptyset$ for $0\le i<j\le m-1$, both
up to sets of $\lambda$-measure zero, and such that the family is
cyclically permuted by translation: $\tau(A_k)= A_{(k+1)\bmod m}$.
Translations act on such partitions via $(\tau A)_k=A_{(k+1)\bmod m}$.
 \par\noindent
 (b) Let $A$ be a $(\lambda,m)$-partition of $S$, and let
$\nu^{(q)}:=q^{-1}\sum_{i=0}^{q-1}\tau^{im}_*\nu$.  Then
$\tilde\mu^{(\lambda,\nu,A)}$ is the state on $\Xha_e\times S$ with
$\tilde\mu^{(\lambda,\nu,A)}(d\n\,d\sigma)
=\mu^{(\nu,A)}_\sigma(d\n)\lambda(d\sigma)$,
where $\mu^{(\nu,A)}_\sigma=\tau_*^k\nu^{(q)}$ iff $\sigma\in A_k$, i.e.,
 \be\label{musigma}
  \mu^{(\nu,A)}_\sigma=\sum_{k=0}^{m-1}\ind_{A_k}(\sigma)\tau_*^k\nu^{(q)}.
 \ee
 $\tilde\mu^{(\lambda,\nu,A)}$ and $\gamma_*\tilde\mu^{(\lambda,\nu,A)}$
will be called {\it basic states}. 
 \par\noindent
 (c) The basic state $\tilde\mu^{(\lambda,\nu,A)}$ is {\it reduced} by
the basic state $\tilde\mu^{(\lambda,\nu,A')}$ if $A'$ is a proper
refinement of $A$; $\tilde\mu^{(\lambda,\nu,A)}$ is then called {\it
reducible}, and we say also that $\gamma_*\tilde\mu^{(\lambda,\nu,A)}$ is
reducible. If $\tilde\mu^{(\lambda,\nu,A)}$ and
$\gamma_*\tilde\mu^{(\lambda,\nu,A)}$ are not reducible then they are
{\it irreducible}.  \end{definition}

Observe that whether or not $\tilde\mu^{(\lambda,\nu,A)}$ is reducible
depends only on $\lambda$, $A$, and $n$, the period of the orbit of $\nu$
under the action of $\tau_*$: $\tilde\mu^{(\lambda,\nu,A)}$ is reducible
precisely when there is a $(\lambda,m')$-partition $A'$ of $S$ such that
$m'$ divides $n$ and $A'$ is a proper refinement of $A$.  Equivalently,
$\tilde\mu^{(\lambda,\nu,A)}$ is irreducible when the action of $\tau^n$
on $A_0$, equipped with the invariant measure $\lambda\big|_{A_0}$, is
ergodic.

We next give some simple consequences of \cdef{answer}.

\begin{lemma}\label{conseq}Let $\tilde\mu^{(\lambda,\nu,A)}$ be a basic
  state.  Then: 
 \par\noindent
(a) $\tilde\mu^{(\lambda,\nu,A)}$ is stationary.
 \par\noindent
 (b) $\tilde\mu^{(\lambda,\tau_*\nu,\tau
A)}(d\n\,d\sigma) =\tilde\mu^{(\lambda,\nu,A)}(d\n\,d\sigma)$.
 \par\noindent
 (c) For any $\sigma\in S$,
$\tau_*\mu^{(\nu,A)}_\sigma=\mu^{(\nu,A)}_{\tau\sigma}
=\mu^{(\nu,\tau^{-1}A)}_\sigma$.
In particular, $\tilde\mu^{(\lambda,\nu,A)}$ is TI. Moreover, if
$\lambda$ has density $\kappa$ then $\tilde\mu^{(\lambda,\nu,A)}$ has
density $\rhohat_e+\kappa$.
 \par\noindent
 (d) If the basic state $\tilde\mu^{(\lambda,\nu,A')}$ reduces
$\tilde\mu^{(\lambda,\nu,A)}$, and $A'_0\subset A_j$, then
$A_i=\bigcup_{k=0}^{p-1}A'_{(i-j+mk)\bmod m}$ for $i=0,\ldots,m-1$, where
$m=|A|$ and $p=|A'|/|A|$.  Moreover,
$\tilde\mu^{(\lambda,\nu,A)}
=p^{-1}\sum_{k=0}^{p-1}\tilde\mu^{(\lambda,\nu,\tau^{-j-mk}A')}$.
 \par\noindent
 (e) If $\tilde\mu^{(\lambda,\nu,A)}$ is irreducible then it cannot be
written as a convex combination of other
$\tilde\mu^{(\lambda',\nu',A')}$; more generally, we cannot have
 \be\label{intconv}
\tilde\mu^{(\lambda,\nu,A)} =\int\tilde\mu^\beta\,\alpha(d\beta),
 \ee
 with $\alpha$ a probability measure on triples
$\beta=(\lambda',\nu',A')$ which assigns zero
probability to $(\lambda,\nu,A)$.  \end{lemma}

\begin{proof}(a), (b), and (c) are immediate.
 \par\noindent
 (d) The first statement follows from
$A_k'=\tau^kA_0'\subset\tau^kA_j=A_{(k+j)\bmod m}$.  For the second it
suffices to prove that
$\mu^{(\nu,A)}_\sigma=p^{-1}\sum_{k=0}^{p-1}\mu^{(\nu,\tau^{-j-mk}A')}_\sigma$
for all $\sigma$, and this follows from the first  and
\eqref{musigma}, by straightforward rearrangement of sums.

  \par\noindent
(e) If \eqref{intconv} holds then ergodicity of $\lambda$ implies that
$\lambda'=\lambda$ $\alpha$-a.s., and thus, $\lambda$-a.s.,
 \be\label{dsig}
\mu_\sigma^{(\nu,A)} =\int\mu_\sigma^\beta\,\alpha(d\beta),
 \ee
 where $\mu_\sigma^\beta:=\mu_\sigma^{(\nu',A')}$.  Now we can use
$\lim_{L\to\infty}(2L+1)^{-1}
\sum_{i=-L}^L\tau_*^i\mu_\sigma^{(\nu,A)}=\nu^{(n)}$,
the corresponding formula for $\mu_\sigma^{(\nu'A')}$, and the bounded
convergence theorem, to conclude from \eqref{dsig} that
$\nu^{(n)}=\int\nu'^{(n')}\,\alpha(d\beta)$ (with $n'$ the period of
$\nu'$).  Moreover, since $\nu$ is extremal, $\nu^{(n)}$ is ETIS. But
then, since $\nu'^{(n')}$ is TIS, $\nu'$ must be a translate of $\nu$
$\alpha$-a.s.

Without loss of generality, then, we may assume that $\nu'=\nu$, and
with that it suffices to show that $A'_0\subset A_0$ (because
$\tau$ acts cyclically on $A$ and all $A'$).  Let $G_i$,
$i=0,\ldots,n-1$, support $\tau_*^i\nu$, with $G_i$ and $G_j$
disjoint for $i\ne j$.  We have from \eqref{intconv} that
 \be
0=\tilde\mu^{(\lambda,\nu,A)}(G_0\times A_0^c)
   =\int\tilde\mu^\beta(G_0\times A_0^c)\,\alpha(d\beta),
 \ee
 so that for $\alpha$-a.e.~$\beta$, $\tilde\mu^\beta(G_0\times A_0^c)=0$
and so also $\tilde\mu^\beta(G_0\times(A_0^c\cap A'_0))=0$.  But for
$\sigma\in A'_0$,
$\mu_\sigma^{(\nu,A')}=\nu^ {(q')}$ (with
$q'=n'/|A'|$), and so
 \be
 0=\tilde\mu^\beta(G_0\times(A_0^c\cap A'_0))
    =\nu^ {(q')}(G_0)\lambda(A_0^c\cap A'_0)
    =\frac1{q'}\lambda(A_0^c\cap A'_0).
 \ee
 Thus $A'_0\subset A_0$.
\end {proof}

\begin{lemma}\label{structure2}If $\Nb_{\rho_e}$ is countable for each
$\rhohat_e\ge1$ then each ETIS state on $\Xha$ is an irreducible basic
state.  \end{lemma}

\begin{proof} We write $\Nb_{\rhohat_e}=\{\nu_j\}_{j\in\bbj}$, with
$\bbj$ either $\bbz$ or a finite set $\{1,2,\ldots,J\}$, and allow the
translation operator $\tau$ to act on $\bbj$ via
$\nu_{\tau j}=\tau_*\nu_j$.  Let $\bbjb$ denote the set of orbits in
$\bbj$ under translation.

 Now let $\mu$ be an ETIS state on $\Xh$ of density $\rhohat$.  Then from
\clem{gamma} we know that $\mu=\gamma_*\tilde\mu$ with
$\tilde\mu(d\n\,d\sigma)=\mu_\sigma(d\n)\lambda(d\sigma)$, $\lambda$
ergodic, $\mu_\sigma$ stationary, and
$\mu_{\tau\sigma}=\tau_*\mu_\sigma$.  For each $\sigma\in S$ we have
$\mu_\sigma=\sum_{j\in\bbj}a_{\sigma,j}\nu_j$, with coefficients
$a_{\sigma,j}$ which, from translation invariance, satisfy
$a_{\tau\sigma,\tau j}=a_{\sigma,j}$.  Now note that if $c\in\bbjb$ then
$Z_\sigma^{(c)}:=\sum_{j\in c}a_{\sigma,j}$ is TI and hence constant
$\lambda$-a.s.; from now on we write this as $Z^{(c)}$.  Define
$\mu_\sigma^{(c)}$ by
$\mu_\sigma^{(c)}:=(Z^{(c)})^{-1}\sum_{j\in c}a_{\sigma,j}\nu_j$ and
$\tilde\mu^{(c)}$ by
$\tilde\mu^{(c)}(d\n\,d\sigma)=\mu^{(c)}_\sigma(d\n)\lambda(d\sigma)$, so
that $\tilde\mu=\sum_{c\in\bbjb}Z^{(c)}\tilde\mu^{(c)}$; since the
$\tilde\mu^{(c)}$ are TIS, the extremality of $\tilde\mu$ implies that
$\tilde\mu=\tilde\mu^{(c_0)}$ for some $c_0\in\bbjb$.

We claim that $c_0$ must be a finite set.  For otherwise we may fix
$j_0\in c_0$ and define $f:S\to\bbz$ by
 \be
f(\sigma):=\min\{i\in\bbz\mid a_{\sigma,\tau^ij_0}
    =\max_{i'\in\bbz}a_{\sigma,\tau^{i'}j_0}\};
 \ee
 $f$ satisfies $f(\tau\sigma)=f(\sigma)+1$, so that by the translation
invariance of $\lambda$ the sets $f^{-1}(\{k\})$, $k\in\bbz$, have equal
$\lambda$-measure, a contradiction.

Thus we are reduced to the case where $c_0$ contains $n$ elements, that is,
if $j_0\in c$ then  $n$ is the minimal period for $\nu_{j_0}$, and 
 \be\label{firstexp}
\mu_\sigma=\sum_{i=0}^{n-1}a_{\sigma,\tau^ij_0}\nu_{\tau^ij_0}.
 \ee
 We regard $a_\sigma:=(a_{\sigma,\tau^ij_0}\bigr)_{i=0}^{n-1}$ as an
element of $\bbr^n$ and let translations act on $\bbr^n$ via
$\tau(b_0,\ldots,b_{n-1})=(b_{n-1},b_0,\ldots,b_{n-2})$, so that
$a_{\tau\sigma}=\tau a_\sigma$. The orbit of $a_\sigma$ under this action
is independent of translations in $\sigma$ and hence constant,
$\lambda$-a.s.; let $\theta$ denote this orbit.

 $|\theta|$ is a positive integer $m$ satisfying $n=qm$ for some integer
$q$.  Let $\theta=\{v_0,\ldots,v_{m-1}\}$, with
$\tau v_k=v_{(k+1)\bmod m}$, let $A_k:=\{\sigma\mid a_\sigma=v_k\}$, and
note that if $w\in\bbr^m$ is defined by $w_i=qv_{0,i}$ then
$w$ is independent of $\sigma$ and $\sum_{i=0}^{m-1}w_i=1$.  Then from
\eqref{firstexp} we find, using \eqref{musigma}, that
$\mu_\sigma=\sum_{r=0}^{m-1}w_i\mu^{(\nu_{j_0},\tau^{-i}A)}_\sigma$.
Thus
 \be
\tilde\mu=\sum_{r=0}^{m-1}w_i\tilde\mu^{(\lambda,\nu_{j_0},\tau^{-i}A)}.
 \ee
 But since $\tilde\mu$ is extremal, precisely one of the $w_i$ can be
 nonzero, so that $\tilde\mu$ is basic.  It then follows from
 \clem{conseq}~(d) that $\tilde\mu$, and hence $\mu$, is irreducible
 basic.\end{proof}

Now we can give the second main result of this appendix.

\begin{theorem}\label{structurex} If $\Nb_{\rho_e}$ is countable for each
$\rhohat_e\ge1$ then the ETIS states on $\Xha$ are precisely the basic
irreducible states.  \end{theorem}

\begin{proof}\clem{structure2} tells us that every ETIS state is an
irreducible basic state.  It follows from this that the decomposition of
any irreducible basic state into ETIS components must be of the form
\eqref{intconv} and hence, by \clem{conseq}~(e), that such a state must
be ETIS.\end{proof}


\begin{thebibliography}{99}

\bibitem{AGLS} A.~Ayyer, S.~Goldstein, J.~L.~Lebowitz, and E.~R.~Speer,
Stationary States of the One-Dimensional Facilitated Asymmetric Exclusion
Process.

\bibitem{bbcs} J. Baik, G. Barraquand, I.
Corwin, and T. Suidan.  Facilitated Exclusion Process. {\it Computation
and Combinatorics in Dynamics, Stochastics and Control}, 1--35, Abel
Symp. {\bf 13}, Springer, Cham, 2018.

\bibitem{BM} Urna Basu and P. K. Mohanty, Active-Absorbing-State Phase
  Transition Beyond Directed Percolation: A Class of Exactly Solvable
  Models.  {\it Phys. Rev. E} {\bf 79}, 041143 (2009).

\bibitem{BESS} Oriane Blondel, Cl\'ement Erignoux, Makiko Sasada, and
Marielle Simon, Hydrodynamic limit for a Facilitated Exclusion Process.
{\it Annales de l'Institut Henri Poincar\'e, Probabilit\'es et
Statistiques} {\bf 56}, 667714 (2020).

\bibitem{BES} Oriane Blondel, Cl\'ement Erignoux, and Marielle Simon,
Stefan Problem for a Non-Ergodic Facilitated Exclusion Process.  {\it
Probability and Mathematical Physics}
{\bf 2}, 127--178 (2021).

\bibitem{CZ} Dayne Chen and Linjie Zhao, The Limiting Behavior of the
  FTASEP with Product Bernoulli Initial Distribution. arXiv:1801.10612v1
  [math PR].

\bibitem{EH} M. R. Evans and T. Hannye, Nonequilibrium Statistical
Mechanics of the Zero-Range Process and Related Models.  {\it J.~Phys.~A:
Math.~Gen.} {\ bf 38} R195 (2005).
 
\bibitem{FPV} P. A. Ferrari, E. Presutti, and M. E. Vares, Local
  Equilibrium for a One-Dimensional Zero Range Process. {\it
    Stoch.~Proc.~Appl} {\bf 26}, 31-45 (1987).

\bibitem{gkr}Alan Gabel, P. L. Krapivsky, and S. Redner, Facilitated
  Asymmetric Exclusion.  {\it Phys. Rev. Lett.}  {\bf 105}, 210603
  (2010).

\bibitem{GR} A. Gabel and S. Redner, Cooperativity-Driven Singularities
in Asymmetric Exclusion, {\it J. Stat. Mech.} {\bf 2011}, P06008 (2011).

\bibitem{GLS1} S. Goldstein, J. L. Lebowitz and E. R. Speer, Exact
Solution of the F-TASEP . J. Stat. Mech. 123202 (2019).

\bibitem{GLS2} S. Goldstein, J. L. Lebowitz and E. R. Speer, The
Discrete-Time Facilitated Totally Asymmetric Simple Exclusion Process.
arXiv:2003.04995 [math-ph]. To appear in {\it Pure and Applied Functional
Analysis}.

\bibitem{GS} Charles M. Grinstead and J. Laurie Snell, {\it Introduction
to Probability, Revised Edition}. American Mathematical Society,
Providence, 2012. 
Available at  {\tt 
https://chance.dartmouth.edu/teaching\_aids/ 
books\_articles/probability\_book/book.html}.

\bibitem{hl}Daniel Hexner and Dov Levine,  Hyperuniformity of
Critical Absorbing States. {\it Phys. Rev. Lett.} {\bf 114}, 110602
(2015).

\bibitem{Kallenberg} Olav Kallenberg, {\it Foundations of Modern
Probability, Second Edition.} Springer, New York, 2002.

\bibitem{mcl}Stefano Martiniani, Paul M. Chaikin, and Dov Levine,
Quantifying Hidden Order out of Equilibrium. {\it Phys. Rev. X} {\bf 9},
011031 (2019).

\bibitem{Oliveira}M\'ario J. de Oliveira,  Conserved Lattice Gas Model with
  Infinitely Many Absorbing States in One Dimension.  {\it Phys. Rev. E} 
  {\bf 71}, 016112 (2005).

\bibitem{rpv}Michela Rossi, Romualdo Pastor-Satorras, and Alessandro
Vespignani, Universality Class of Absorbing Phase Transitions with a
Conserved Field. {\it Phys. Rev.  Lett.} {\bf85}, 1803 (2000).

\bibitem{ST}Vladas Sidoravicius and Augusto Teixeira, Absorbing-State
Transition for Stochastic Sandpiles and Activated Random Walks. {\it
Electron.  J. Probab.} {\bf 22}, no. 33, 1–35 (2017).

\end{thebibliography}
\end{document}